\documentclass[english]{article}
\usepackage[T1]{fontenc}
\usepackage[latin9]{inputenc}
\usepackage{amsmath}
\usepackage{amsthm}
\usepackage{amssymb}
\usepackage{cancel}
\usepackage{stmaryrd}
\usepackage{graphicx}
\usepackage{wasysym}

\makeatletter

\let\SF@@footnote\footnote
\def\footnote{\ifx\protect\@typeset@protect
    \expandafter\SF@@footnote
  \else
    \expandafter\SF@gobble@opt
  \fi
}
\expandafter\def\csname SF@gobble@opt \endcsname{\@ifnextchar[
  \SF@gobble@twobracket
  \@gobble
}
\edef\SF@gobble@opt{\noexpand\protect
  \expandafter\noexpand\csname SF@gobble@opt \endcsname}
\def\SF@gobble@twobracket[#1]#2{}
\providecommand{\tabularnewline}{\\}

\theoremstyle{plain}
\newtheorem{thm}{\protect\theoremname}
\theoremstyle{plain}
\newtheorem{prop}[thm]{\protect\propositionname}
\ifx\proof\undefined
\newenvironment{proof}[1][\protect\proofname]{\par
	\normalfont\topsep6\p@\@plus6\p@\relax
	\trivlist
	\itemindent\parindent
	\item[\hskip\labelsep\scshape #1]\ignorespaces
}{%
	\endtrivlist\@endpefalse
}
\providecommand{\proofname}{Proof}
\fi
\newenvironment{lyxlist}[1]
	{\begin{list}{}
		{\settowidth{\labelwidth}{#1}
		 \setlength{\leftmargin}{\labelwidth}
		 \addtolength{\leftmargin}{\labelsep}
		 }}
	{\end{list}}

\makeatother

\usepackage{babel}
\providecommand{\propositionname}{Proposition}
\providecommand{\theoremname}{Theorem}

\begin{document}
\title{\date{}General approach to function approximation}
\author{Andrej Liptaj\thanks{andrej.liptaj@savba.sk, ORC iD 0000-0001-5898-6608}\\
{\small{}Institute of Physics, Bratislava, Slovak Academy of Sciences}}
\maketitle
\begin{abstract}
Having a function $f$ and a set of functionals $\{\mathcal{C}_{n}\}$,
$c_{n}^{f}\equiv\mathcal{C}_{n}\left(f\right)$, one can interpret
function approximation very generally as a construction of some function
$\mathcal{A}^{f}$ such that $c_{n}^{f}=\mathcal{C}_{n}\left(\mathcal{A}^{f}\right)$.
All known approximations can be interpreted in this way and we review
some of them. In addition, we construct several new expansion types
including three rational approximations.
\end{abstract}


\section{Introduction}

Applications of function approximations are countless: evaluation
of functions on computers, transformation of functions into a form
suited for further processing (integration, differentiation), frequency
analysis, (approximate) solutions of differential equations, etc.
A natural point of view which is common to all approximation methods
is that by an approximation of a function $f$ one can understand
another function $\mathcal{A}^{f}$ which meets a number of constraints
(conditions) true for $f$. $\mathcal{A}^{f}$ is usually constructed
such as to keep some well defined function form and the constraints
are met by adjusting coefficients appearing within this form. The
number of constraints can be finite or infinite, and in most cases
the degree of the approximation can be increased by increasing the
number of conditions which are obeyed. For an exact approximation
$\mathcal{A}^{f}=f$ to be achieved, one in general expects an infinite
number of fulfilled constraints as a necessary (but often not sufficient)
condition. This can be formalized in a framework where requirements
are represented by numbers resulting from the action of some predefined
set of functionals on $f$. Then the construction of an approximation
corresponds to tuning the above-mentioned coefficients, such as to
reproduce these numbers by $\mathcal{A}^{f}$. As an example we can
mention the Taylor series which are based on matching the derivatives
by power series, or the Fourier expansion which can be seen as matching
the integrals where the integrand is $f\left(x\right)$ multiplied
by $\sin\left(nx\right)$ or $\cos\left(nx\right)$.

These ideas are not new, similar ideas have been presented in \cite{Widder1928,davis1975interpolation,Iran1_2008JCoAM.216..307M,Iran2_article,Iran3_article}.
To justify our work we bring forward two points to make a distinction.
\begin{itemize}
\item First point is conceptual: all previous works are based on linear
functionals\footnote{We have to honestly acknowledge that a possibility of non-linear functionals
is indirectly mentioned in \cite{davis1975interpolation} where in
Sec. 1.12 the author says: ``Interpolation theory is concerned with
reconstructing functions on the basis of certain functional information
assumed known. In many cases, the functionals are linear.''}. The idea of linearity is indeed attractive when practical considerations
are taken into account, but is not necessary to define the concept
of approximation. Functionals need to constrain the function in the
first place, not to act linearly. We thus extend the definition of
approximation in this direction and provide a simple example of this
(i.e. non-linear) type.
\item The second point is practical: we construct several new expansions
not present in other works.
\end{itemize}
The introduced notion of an approximation does not imply the convergence
(of any type) of \emph{$\mathcal{A}^{f}$} to $f$. In most of this
text we focus only on the approximation properties and not on the
convergence ones. We do this because (inferring from the existing
approximations\footnote{The Padé approximants come to the author's mind.})
the convergence questions are usually technically complicated and
their study (weakness/strength, sufficiency/necessity, bounding strategies,
etc.) often represents a large topic which would exceed the intended
extent of this text, where many different approximation types are
mentioned.

Also, our text does not rely on proving a single main idea or theorem.
Going through various function approximations a couple of proofs are
presented labeled as propositions.

In Sec. \ref{sec:General-approach-to} we introduce the notion of
an approximation more formally and discuss some basic properties.
The section which follows then reviews some of the existing examples
and shows how to interpret them in the general approximation approach.
We present new function expansions in Sec. \ref{sec:newExamples}
and close the text by summary, conclusion and outlook.

\section{General approach to function approximation\label{sec:General-approach-to}}

\subsection{Definitions}

Let $f\left(x\right)$ be a real-valued function defined on some (open
or closed) interval $\left(a,b\right)\subset\overline{\mathbb{R}}$
and let $\left\{ \mathcal{C}_{n}\right\} _{n=0}^{\infty}$ be a sequence
of real functionals acting on $f$ 
\[
\mathcal{\mathcal{C}}_{n}:f_{\left(a,b\right)}\left(x\right)\rightarrow\mathbb{R}.
\]
By \emph{characteristic numbers} of the function $f$ we will understand
the sequence of real numbers $\left\{ c_{n}^{f}\right\} _{n=0}^{\infty}$
defined by 
\[
c_{n}^{f}\equiv\mathcal{C}_{n}\left(f\right).
\]
By a \emph{partial approximation} of the function $f$ of the order
$N$ we will understand any function $\mathcal{A}_{N}^{f}\left(x\right)$
for which the action of $\mathcal{\mathcal{C}}_{n}$ is defined for
all $n\leq N$ having the property
\[
\mathcal{\mathcal{C}}_{n}\left(\mathcal{A}_{N}^{f}\right)=c_{n}^{f}\text{ for all }n\leq N.
\]
A function $\mathcal{A}^{f}\left(x\right)\equiv\mathcal{A}_{\infty}^{f}\left(x\right)$
will be called the \emph{approximation} of the function $f$.

\subsection{Properties of functionals and approximations}

Intuitively, some basic properties of functionals are expected. We
call the functional $\mathcal{\mathcal{C}}_{i}\,\in\,\mathcal{C}\equiv\left\{ \mathcal{C}_{n}\right\} _{n=0}^{\infty}$
dependent on $\mathcal{C}\backslash\mathcal{\mathcal{C}}_{i}$ if
for any $f$ (for which the functionals are well defined) $c_{i}^{f}$
can be expressed as a function of the remaining characteristic numbers
\[
c_{i}^{f}=c_{i}^{f}\left(c_{0}^{f},\ldots,c_{i-1}^{f},c_{i+1}^{f},\ldots\right).
\]
For the set $\mathcal{C}$ to be rich enough to permit an approximation
of an interestingly large family of functions, one expects it to contain
an infinite number of independent functionals. For aesthetic reasons
one may prefer all functionals in $\mathcal{C}$ to be independent.

Also, in situations where the functional action needs to be distributed
over infinite sums, the functionals are assumed to be continuous.
Yet, because our approach is very general, we prefer to \emph{assume}
that the various operations we perform are valid, i.e. that the appearing
objects have the necessary properties (whatever these are), rather
then specifying conditions for them.

\subsubsection{Linearity}

Most of the actually used approximations are based on linear functionals,
i.e. one has 
\[
\mathcal{C}_{n}\left(f+\alpha g\right)=\mathcal{C}_{n}\left(f\right)+\alpha\mathcal{C}_{n}\left(g\right)\text{ for all }f,g\text{ and }n\,\in\,\mathbb{N}_{0},\,\alpha\,\in\,\mathbb{R}.
\]
Such approximations are very appealing if a set of functions $\left\{ \Delta_{n}\right\} _{n=0}^{\infty}$
with the delta property
\begin{equation}
\mathcal{C}_{n}\left(\Delta_{m}\right)=\delta_{n,m}\label{eq:deltaFnDef}
\end{equation}
exists. When so, an approximation of $f$ can be easily constructed
in the form of an infinite sum
\begin{equation}
\mathcal{A}^{f}\left(x\right)=\sum_{m=0}^{\infty}c_{m}^{f}\Delta_{m}\left(x\right),\label{eq:LinApprox}
\end{equation}
where the existence of the limit is assumed. Assuming further that
the action of $\mathcal{\mathcal{C}}_{i}$ can be distributed over
this infinite sum, one observes that the expression (\ref{eq:LinApprox})
indeed reproduces the characteristic numbers of $f$
\[
\mathcal{C}_{n}\left(\mathcal{A}^{f}\right)=\sum_{m=0}^{\infty}c_{m}^{f}\mathcal{C}_{n}\left(\Delta_{m}\right)=\sum_{m=0}^{\infty}c_{m}^{f}\delta_{n,m}=c_{n}^{f}.
\]
In this text we will refer to $\Delta_{n}$ using the term ``delta
function''.

Another interesting scenario, which also appears in practice, is represented
by the set of functions $\left\{ \nabla_{n}\right\} _{n=0}^{\infty}$
with the property
\begin{equation}
c_{n}^{\nabla_{m}}\equiv\mathcal{C}_{n}\left(\nabla_{m}\right)=0\text{ for all }n<m.\label{eq:triangularPropertyDef}
\end{equation}
Also in this case one can propose to build an approximation as a series
with multiplicative coefficients (supposing it is well defined)
\begin{equation}
\mathcal{A}^{f}\left(x\right)=\sum_{m=0}^{\infty}t_{m}\nabla_{m}\left(x\right),\label{eq:TriangApproach}
\end{equation}
where the question of $t_{m}$ values arises. Assuming the distribution
of the functional action is justified, one has
\[
\mathcal{C}_{n}\left(\mathcal{A}^{f}\right)=\sum_{m=0}^{\infty}t_{m}\mathcal{C}_{n}\left(\nabla_{m}\right)=\sum_{m=0}^{n}t_{m}c_{n}^{\nabla_{m}}\overset{!}{=}c_{n}^{f},
\]
where the notation $\overset{!}{=}$ reads ``should be equal to''.
The relation between $t_{m}$ and $c_{n}^{f}$ is therefore represented
by an infinite triangular matrix 
\[
\boldsymbol{c^{f}}=\boldsymbol{Tt},\quad\boldsymbol{T}=\left(\begin{array}{cccc}
c_{0}^{\nabla_{0}} & 0 & 0 & \ldots\\
c_{1}^{\nabla_{0}} & c_{1}^{\nabla_{1}} & 0 & \ldots\\
c_{2}^{\nabla_{0}} & c_{2}^{\nabla_{1}} & c_{2}^{\nabla_{2}} & \ldots\\
\vdots & \vdots & \vdots & \ddots
\end{array}\right),
\]
which can be easily inverted in practice (and sometimes also formally)
$\boldsymbol{t}=\left(\boldsymbol{T}\right)^{-1}\boldsymbol{c^{f}}$,
thus allowing for an easy-to-achieve approximation. We will refer
to $\nabla_{n}$ using the term ``triangular function''.

An additional issue, which can be addressed when series of delta functions
are used, is the behavior of the latter with the number of matched
constraints increasing to infinity. Here we naturally extend the definition
of a delta function to partial approximations by
\begin{equation}
\mathcal{C}_{n}\left(\Delta_{m}^{[N]}\right)=\delta_{n,m}\text{ for }n,m\leq N,\label{eq:partialDeltaFns}
\end{equation}
where $[N]$ indicates the approximation order. If functions $\Delta_{m}$
exist then they also represent partial delta functions, i.e. $\Delta_{m}^{[N]}$
exist too. However, since $\Delta_{m}^{[N]}$ functions need to satisfy
only a finite number of conditions (i.e. conditions for their existence
are weaker) one can observe situations
\begin{itemize}
\item where $\Delta_{m}^{[N]}$ exist, but $\Delta_{m}$ do not, or
\item where, for fixed $N$, several realizations of $\Delta_{m}^{[N]}$
exist, with $\Delta_{m}$ being one of them.
\end{itemize}
Similarly one can define partial triangular functions by
\[
\mathcal{C}_{n}\left(\nabla_{m}^{[N]}\right)=0\text{ for all }n<m\leq N
\]
and address analogous questions.

\subsubsection{Construction of approximations}

One can think of multiple ways of how an approximation can be built.
It is natural to expect from its function form to have some generality,
i.e. to be suitable for approximating various functions, while keeping
its form. Thus one intuitively comes to the following ideas:
\begin{itemize}
\item The form of an (partial) approximation has some well defined logical
structure, which remains the same for all approximated functions.
\item The (partial) approximation is achieved by varying coefficients $a_{i}$
appearing within this fixed form.
\end{itemize}
Here, once more, some general properties are expected, such as the
number of coefficients to adjust. If a partial approximation $\mathcal{A}_{N}^{f}$
is supposed to reproduce $N+1$ characteristic numbers, it is reasonable
to assume this can be achieved by introducing $N+1$ coefficients
$\mathcal{A}_{N}^{f}=\mathcal{A}_{N}^{f}\left(a_{0},\ldots,a_{N}\right)$.
Of course, the latter cannot be guaranteed, however, the approximations
chosen by humans always have this property. One also expects that
all coefficients play a role (variation of each coefficient has some
impact on $c_{n}^{\mathcal{A}^{f}}$ values), and their actions are
independent (variation of some coefficient cannot be replaced by variations
of other coefficients).

We will now adopt this natural choice and work with the scenario where
the number of coefficients for $\mathcal{A}_{N}^{f}$ is $N+1$
\begin{align*}
\mathcal{A}_{0}^{f} & =\mathcal{A}_{0}^{f}\left(a_{0}^{[0]}\right),\\
\mathcal{A}_{1}^{f} & =\mathcal{A}_{1}^{f}\left(a_{0}^{[1]},a_{1}^{[1]}\right),\\
\mathcal{A}_{2}^{f} & =\mathcal{A}_{2}^{f}\left(a_{0}^{[2]},a_{1}^{[2]},a_{2}^{[2]}\right),\\
 & \vdots\quad.
\end{align*}
Here we introduce the notation $a_{i}^{[N]}$ meaning that the coefficient
is one of those appearing in the formula for $\mathcal{A}_{N}^{f}$.

An important property related to the technical complexity of an approximation
is the persistence of coefficient values across different approximation
orders, since it is very convenient to keep the coefficient value
once computed at some order of approximation also for the next ones,
$a_{0}^{[0]}=a_{0}^{[1]}=a_{0}^{[2]}=\ldots=a_{0}$. Such a property
can be observed in most approximations used in practice and as one
of rare counterexamples one can mention the Padé approximant, where,
for each approximation order, a new set of coefficients has to be
determined. We will label the approximations with the coefficient
persistence

\begin{align*}
\mathcal{A}_{0}^{f} & =\mathcal{A}_{0}^{f}\left(a_{0}\right),\\
\mathcal{A}_{1}^{f} & =\mathcal{A}_{1}^{f}\left(a_{0},a_{1}\right),\\
\mathcal{A}_{2}^{f} & =\mathcal{A}_{2}^{f}\left(a_{0},a_{1},a_{2}\right),\\
 & \vdots\quad.
\end{align*}

as triangular, since they are naturally (but not necessarily) realized
in the linear-functional approach by combining triangular functions
(formula (\ref{eq:TriangApproach})). As a subset of these, the most
elegant approximations are those where the coefficients are in the
one-to-one correspondence with characteristic numbers, i.e. there
is a single, particular coefficient to be tuned to match a given characteristic
number without interfering with others
\[
\mathcal{C}_{n}\left(\mathcal{A}^{f}\right)\equiv c_{n}^{\mathcal{A}^{f}}\left(a_{n}\right)\overset{!}{=}c_{n}^{f}.
\]
We will call such approximations as delta approximations. They are
in practice often realized in the linear-functional framework as linear
combination of delta functions (formula (\ref{eq:LinApprox})), they
are, however, not restricted to this scenario (as presented later).

To achieve the characteristic number matching in a general case, one
usually has to solve a system of $N+1$ (in general non-linear) equations
for coefficients $a_{0}^{[N]},\ldots,a_{N}^{[N]}$ 
\begin{align}
\mathcal{\mathcal{C}}_{0}\left(\mathcal{A}_{N}^{f}\left(a_{0}^{[N]},\ldots,a_{N}^{[N]}\right)\right) & =c_{0}^{f},\nonumber \\
 & \vdots\label{eq:solvingGenetalEquations}\\
\mathcal{\mathcal{C}}_{N}\left(\mathcal{A}_{N}^{f}\left(a_{0}^{[N]},\ldots,a_{N}^{[N]}\right)\right) & =c_{N}^{f}.\nonumber 
\end{align}
For triangular approximations one typically solves $N+1$ single equations
in a progressive way
\begin{equation}
\mathcal{\mathcal{C}}_{N}\left(\mathcal{A}_{N}^{f}\left(a_{N}\right)\right)=c_{N}^{f},\label{eq:progressiveSolution}
\end{equation}
where the values of $a_{0},\ldots,a_{N-1}$ are known from solving
similar equations in previous steps. For delta approximations the
solution needs not to be progressive since equations (\ref{eq:progressiveSolution})
become independent. Furthermore, in the latter scenario, the equations
have often a similar structure and can be all formally solved in a
single step.

\section{Existing examples}

The overview which follows presents some of the existing approximation
approaches in the light of the scheme developed in the previous section.
For each method we try to provide a structured entry where its basic
features are summarized.

\subsection{Derivative matching}

The derivative-matching approximations are based on linear functionals
\[
\mathcal{C}_{n}=\frac{d^{n}}{dx^{n}}|_{x=x_{0}},\quad c_{n}^{f}=\frac{d^{n}f\left(x\right)}{dx^{n}}|_{x=x_{0}},
\]
where the point of the differentiation $x_{0}$ needs to be specified.
In the four examples which follow it is natural to build the approximation
around $x_{0}=0$. For a different point $x_{0}\neq0$ a shifted version
of the approximation is used $\mathcal{A}^{f}=\mathcal{A}^{f}\left(x-x_{0}\right)$,
so that the argument becomes zero for $x=x_{0}$. Thus, for the simplicity
and without loss of generality, we choose in the four examples $x_{0}=0$;
for any other value the expansion can be shifted by shifting the argument.

\subsubsection*{Taylor series}

The Taylor series $\mathcal{A}^{f,T}$ is a delta approximation based
on delta functions
\[
\mathcal{A}^{f,T}\left(x\right)=\sum_{n=0}^{\infty}a_{n}\Delta_{n},\quad\Delta_{n}=\frac{1}{n!}x^{n}
\]
with coefficients
\[
a_{n}=c_{n}^{f}.
\]
A large body of literature covers the theory related to the Taylor
series and one might consult it for information about the convergence
behavior and the related criteria. The function $f$ having a converging
approximation $\mathcal{A}^{f,T}\rightarrow f$ in a non-zero neighborhood
of $x_{0}$ is called analytic at $x_{0}$. The notion of analyticity
plays an important role in mathematics, especially when function arguments
are extended to complex numbers. The Taylor expansions are very popular
because they can be easily manipulated (integrated, differentiated,
computed) and they are especially useful when small perturbations
(of any kind) are studied.

The approximation needs not to converge for a real function even if
all characteristic numbers (derivatives) exist and are matched, there
are well known examples of non-analytic smooth functions. This is
different in the complex analysis, where the existence of a continuous
derivative implies the analyticity.

\subsubsection*{Neumann series of Bessel functions}

The Neumann series of Bessel functions (NsBf) $\mathcal{A}^{f,N}$\cite{NeumannBook,WatsonBook}
is a triangular approximation based on triangular functions
\[
\mathcal{A}^{f,N}\left(x\right)=\sum_{n=0}^{\infty}a_{n}\nabla_{n},\quad\nabla_{n}=J_{n}(x),
\]
where $J_{n}$ are the Bessel functions of the first kind. The coefficients
are given by
\[
a_{n}=\begin{cases}
c_{0}^{f} & \text{for }n=0\\
\sum_{i=0}^{2i<n}2^{n-2i}\left[\binom{n-i-1}{n-2i-1}+2\binom{n-i-1}{n-2i}\right]c_{n-2i}^{f} & \text{for }n>0
\end{cases},
\]
with brackets denoting the binomial coefficients. The convergence
properties are known to be the same as for the Taylor series with
identical characteristic numbers (\cite{WatsonBook}, 16.2, ``Pincherle's
theorem''), thus the applicability domain of this expansion corresponds
to analytic functions.

The NsBf are seen more rarely, they play a role when studying Bessel's
differential equation or some similar equations, such as presented
in \cite{usedBessel1,usedBessel2}. The numerical test (Fig. (\ref{pic:BessCos}))
also suggests the NsBf significantly over-perform the Taylor series
when periodic functions are approximated.
\begin{figure}
\begin{centering}
\includegraphics[width=0.6\linewidth]{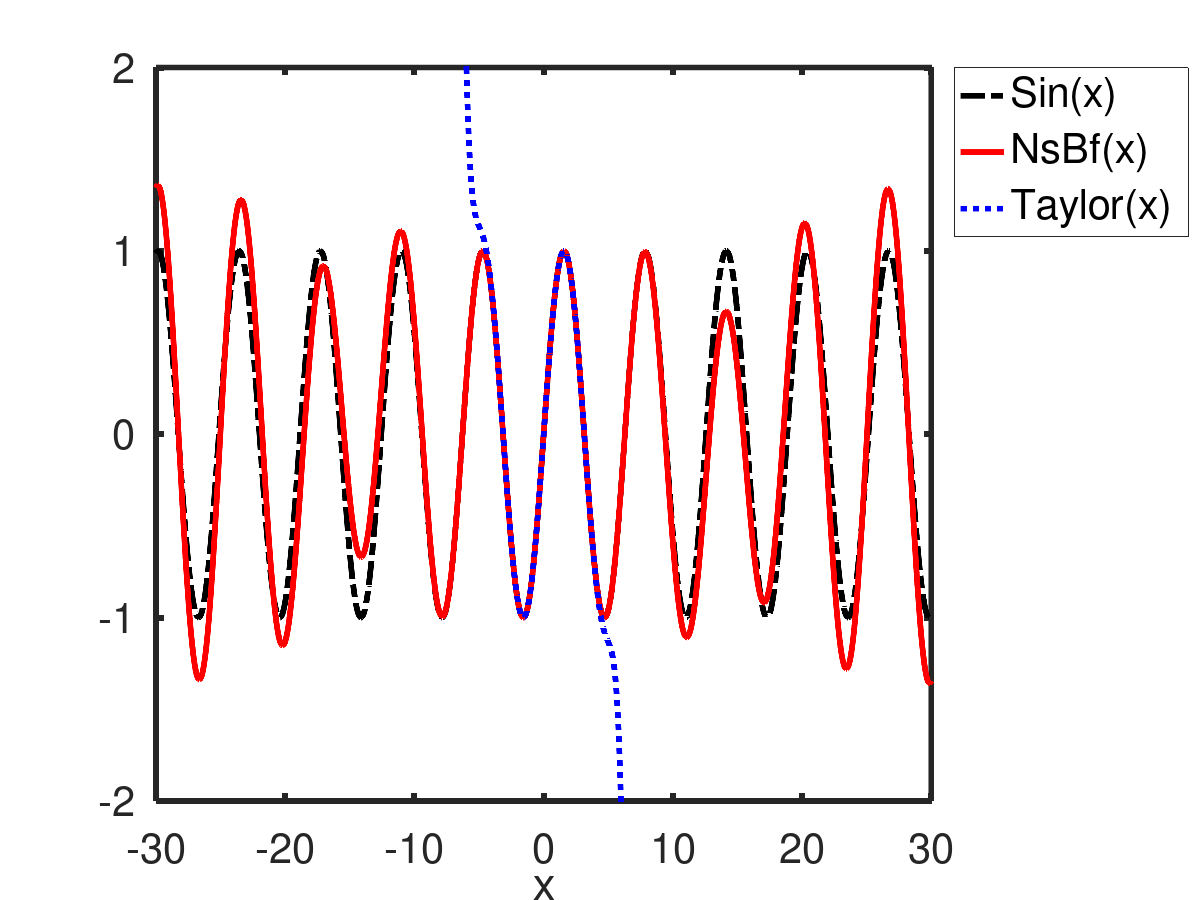}
\par\end{centering}
\caption{Approximation of $y=\sin\left(x\right)$ by the NsBf and by the Taylor
series with the value of the function and values of the first ten
derivatives matched.}
\label{pic:BessCos}
\end{figure}

\subsubsection*{Approximation of Padé}

The Padé partial approximation by a rational function is a rather
well-known example which does not fall into any category of Sec. \ref{sec:General-approach-to}.
The approximation $\mathcal{A}_{N}^{f,P}$ is build as a ratio of
two polynomials
\[
\mathcal{A}_{m+n+1}^{f,P}(x)=\frac{P_{m}(x)}{P_{n}(x)},\:P_{m}(x)=\sum_{i=0}^{m}a_{1,i}^{[m+n+1]}x^{i},\:P_{n}(x)=1+\sum_{j=1}^{n}a_{2,j}^{\left[m+n+1\right]}x^{j},
\]
where the absolute term of $P_{n}$ is (by definition) one. The approximation
is, to the author's knowledge, always used as partial, which may be
understood as the consequence of its complicated character (non-persistent
coefficients). The coefficients $a_{1,i}^{[m+n+1]}$, $a_{2,i}^{[m+n+1]}$
can be determined in a ``brute-force'' manner by differentiating
$\mathcal{A}_{m+n+1}^{f,P}$ and solving the resulting equations,
even though more efficient methods are on the market, see \cite{baker1996pade}.
The latter reference corresponds to one of the large number of texts
which cover Padé approximation and where further information (e.g.
about the convergence issues) can be found.

The Padé approximants are known to have good convergence properties,
in some situations the domain of the convergence extends beyond that
of the Taylor series. Often, the Padé approximation converges more
rapidly than the Taylor series which explains its popularity in numerical
computations \cite{padeNum}. These appealing features are attributed
to the fact that the form of rational functions enables the approximation
to mimic the poles (and branch cuts \cite{padeBranch}) of the approximated
functions.

We discuss further this topic also in Sections \ref{subsec:Power-series-of}
and \ref{subsec:g(x^n)} where we introduce presumably new ways of
constructing a rational approximation.

\subsubsection*{Powers of sines\label{subsec:Powers-of-sines}}

The derivative-matching approach can be applied also to trigonometric
polynomials \cite{Butzer}. One has
\begin{equation}
\mathcal{A}_{n}^{f,S}(x)=a_{0}+\sum_{n=1}^{\infty}a_{n}\left[\sin\left(\frac{x}{2}\right)\right]^{n}\label{eq:powersOfSine}
\end{equation}
with
\[
a_{0}=c_{0}^{f},\quad a_{n}=\frac{2^{n}}{n!}\sum_{k=1}^{n}c_{k}^{f}\left|t(n,k)\right|
\]
where $t(n,k)$ are the central factorial numbers of the first kind
defined by the equation\footnote{A closed formula is not available \cite{tjant2019725}.}
\[
x\left(x+\frac{n}{2}-1\right)\left(x+\frac{n}{2}-2\right)\ldots\left(x-\frac{n}{2}+1\right)=\sum_{k=0}^{n}t\left(n,k\right)x^{k},\quad n\epsilon\mathbb{N}.
\]
The expansion is certainly interesting, the use of trigonometric power
formulas allows one to transform (\ref{eq:powersOfSine}) into a formal
Fourier series on $\left(-2\pi,2\pi\right)$, yet not constructed
by scalar products but upon ``Taylor's'' principles. The authors
of \cite{Butzer} study in details the convergence properties of this
series and also claim that such an expansion may be useful, e.g. in
the theory of trigonometric asymptotic expansions. The approximation
(\ref{eq:powersOfSine}) is triangular and represents a special case
of a larger group of approximations which we address in Sec. \ref{subsec:Power-series-of}.

\subsection{Matching re-weighted\protect\footnote{The expression ``re-weigh function'' we use is chosen to make a
distinction from ``weight function'', which could one interpret
as having a unit integral $\int_{a}^{b}dx\,v_{n}\left(x\right)=1$
which we do \emph{not} ask for.} integrals}

A large family of approximations is based on integrals where the integrand
is the product of the approximated function and a function which belongs
to some predefined (usually infinite) function set
\begin{equation}
\mathcal{C}_{n}=\int_{a}^{b}dx\,\left[v_{n}\left(x\right)\times\ldots\right]\quad c_{n}^{f}=\int_{a}^{b}v_{n}\left(x\right)f\left(x\right)\,dx,\quad v_{n}\epsilon V\equiv\left\{ v_{i}\right\} _{i=0}^{\infty}.\label{eq:rewIntegralsDef}
\end{equation}
Most of these approximations are usually formally interpreted in terms
of a vector space of functions where the integral defines the scalar
product and $V$ represents a complete orthonormal basis. There are
exceptions, e.g. such an interpretation is not adopted for the moment
problem, yet the (raw) moments are also computed as re-weighted function
integrals.

The vector-space based approximations are numerous, extend to more
dimensions (e.g. the spherical harmonics $Y_{l}^{m}\left(\theta,\varphi\right)$)
with many different basis, thus the following examples should be understood
as basic illustrative examples for this group of approximations. The
orthonormality of the basis corresponds to the delta property of the
approximation.

\subsubsection*{Fourier series and generalized Fourier series}

Let us in parallel present the two most common ``vector-space''
approximations, the Fourier series ($\mathcal{A}^{f,F}$) and the
Legendre-Fourier series ($\mathcal{A}^{f,L}$) defined on their usual
intervals $\left(-\pi,\pi\right)$ and $\left(-1,1\right)$ respectively.
One has
\[
\mathcal{A}^{f,F}\left(x\right)=\sum_{n=0}^{\infty}a_{n}v_{n},\quad v_{n}=\Delta_{n}=\begin{cases}
\frac{1}{\sqrt{2}} & \text{for }n=0\\
\sin\left(\frac{n+1}{2}x\right) & \text{for }n=1,3,5,\ldots\\
\cos\left(\frac{n}{2}x\right) & \text{for }n=2,4,6,\ldots
\end{cases},
\]
\[
\mathcal{A}^{f,L}\left(x\right)=\sum_{n=0}^{\infty}a_{n}v_{n},\quad v_{n}=\Delta_{n}=\sqrt{\frac{2}{2n+1}}P_{n},
\]
where $P_{n}$ are the Legendre polynomials. Both approximations are
delta approximations based on linear functionals and delta functions.
The relation to characteristic numbers stands
\begin{equation}
a_{n}=c_{n}^{f}.\label{eq:momentProbDef}
\end{equation}
This approach generalizes to other expansions, such as Fourier--Bessel
series, series of spherical harmonics (Laplace series), Schlömilch's
series, etc. Since $\left\{ v_{i}\right\} _{i=0}^{\infty}$ represents
a complete basis of a normed vector space, the series convergence
in the norm ($L^{p}$) implies an almost everywhere pointwise convergence.

The Fourier series and their generalizations are very useful in many
different areas of mathematics and physics, among the most important
are the frequency analysis or solution techniques for differential
equations. The approximation by Fourier series is limited to periodic
functions.

\subsubsection*{Moments}

The raw moments are defined by 
\[
v_{n}=x^{n}.
\]
We conjecture that the delta functions cannot be constructed and prove
it for stronger assumptions hereunder. The construction of the approximation
$\mathcal{A}^{f,M}$ from the characteristic numbers (i.e. the raw
moments) is a well-known moment problem \cite{alma991004750549703276}.
The approach we present is however more general (than the classical
moment problem), since, in the latter, one usually assumes $f$ to
be positive\footnote{Often interpreted as a measure, a probability function, or a mass
distribution.}, which we do not do here. Depending on the interval, the moment problem
is labeled as \emph{Hamburger} ($a=-\infty,\:b=\infty$), \emph{Stieltjes}
($a=0,\:b=\infty$) or \emph{Hausdorff} (both $a$ and $b$ finite).
For illustration purposes we choose the latter and (without loss of
generality\footnote{A complete system of orthogonal functions on $\left(-1,1\right)$
can be scaled to arbitrary interval $\left(a,b\right)$ by linearly
scaling the argument of the functions.}) assume $a=-1$, $b=1$ (and a general, possibly negative $f$).
In this context we prove
\begin{prop}
The set $\left\{ \Delta_{n}\right\} _{n=0}^{\infty}$ cannot be constructed
from continuous functions expressible as Fourier-Legendre series.\label{prop:notDeltaForMoments}
\end{prop}
\begin{proof}
We proceed by contradiction and assume the existence of such $\Delta_{n}$
for all $n\epsilon\mathbb{N}_{0}$
\begin{align*}
\Delta_{n}(x) & =\sum_{i=0}^{\infty}\alpha_{i}^{n}L_{i}(x),\quad\alpha_{i}^{n}=\int_{-1}^{1}\Delta_{n}(x)L_{i}(x)dx=2^{i}\sqrt{\frac{2}{2i+1}}\binom{i}{n}\binom{\frac{i+n-1}{2}}{i},
\end{align*}
where $\left\{ L_{i}\right\} $ is an orthonormal basis $L_{i}=\sqrt{\frac{2}{2i+1}}P_{i}$
($P_{i}$ being the Legendre polynomials), the brackets denote the
generalized form of the binomial coefficient and we use (after writing
$L_{i}(x)$ as polynomial) the assumed delta property of $\Delta_{n}$.
Let us investigate the value of $\Delta_{n}$ at $x_{0}=0$ for some
even $n=2m$. One has
\begin{align*}
\Delta_{2m}(0) & =\sum_{j=0}^{\infty}\alpha_{j}^{2m}L_{j}(0)=\sum_{j=0}^{\infty}\left(-1\right)^{j}\frac{2}{4j+1}\binom{2j}{2m}\binom{j+m-\frac{1}{2}}{2j}\binom{2j}{j}\equiv\sum_{j=0}^{\infty}q_{j}^{m}.
\end{align*}
Re-writing an individual sum element in terms of factorials (using
the factorial expressions for $\Gamma\left[\frac{1}{2}\pm n\right]$)
one can study the large-$j$ behavior ($j\geq m$)
\begin{align*}
q_{j}^{m} & =\frac{\left(-1\right)^{m}}{4^{2j}}\:\frac{2}{4j+1}\:\frac{\left(2j\right)!}{\left(2m\right)!\left(j!\right)^{2}}\:\frac{\left(2j+2m\right)!}{(j+m)!(j-m)!}.
\end{align*}
The limit $j\rightarrow\infty$ is easily determined by taking the
logarithm and using the Stirling's approximation of the factorial
\[
\lim_{j\rightarrow\infty}\ln\left[\left(-1\right)^{m}q_{j}^{m}\right]=\infty.
\]
Thus $\Delta_{n}$ is (at zero and at least for some $n$) either
not defined (i.e. continuous) or not expressible as Fourier-Legendre
series. Both cases contradict the assumptions.
\end{proof}
Continuous partial delta functions $\Delta_{m}^{[N],M}$ can be found
by assuming (for example) a polynomial form and by solving equations
(\ref{eq:partialDeltaFns}), or by combining the Legendre polynomials
in an appropriate way so that the partial-delta behavior is obtained
(as demonstrated in the next paragraph). Numerical results suggest
that the expressions for $\Delta_{n}^{\left[N\rightarrow\infty\right]}$
diverge\footnote{If constructed as polynomials of a minimal degree, the partial delta
functions become with increasing $N$ more and more oscillatory with
an importantly rising amplitude.} almost everywhere on $\left(-1,1\right)$.

The moment-matching problem can be solved using a triangular approximation,
one example of such is represented by the Fourier-Legendre series
\cite{ASKEY1982237,Talenti_1987,2017CMMPh..57..175S}. The Legendre
polynomials $P_{m}=\nabla_{m}$ obey the triangular property (\ref{eq:triangularPropertyDef})
\begin{equation}
\int_{-1}^{1}P_{m}x^{n}dx=0\text{ for all }n<m\label{eq:LegPolyOrthog}
\end{equation}
and any partial approximation of $f$ of the order $N$ matches its
first $N+1$ moments. Indeed, if $f$ is expanded on $(-1,1)$ into
the series
\begin{equation}
f\left(x\right)=\sum_{n=0}^{N}\beta_{n}P_{n}\left(x\right),\label{eq:mometTriangularLegendre}
\end{equation}
then $m$th moment of the latter ($0\leq m\leq N$) is given by
\begin{align*}
 & \int_{-1}^{1}dx\:x^{m}\sum_{n=0}^{N}\beta_{n}P_{n}\left(x\right)=\\
 & \quad=\int_{-1}^{1}dx\:x^{m}\sum_{n=0}^{N}\left(\frac{2n+1}{2}\int_{-1}^{1}dy\:f\left(y\right)P_{n}\left(y\right)\right)P_{n}\left(x\right),\\
 & \quad=\int_{-1}^{1}dx\:x^{m}\int_{-1}^{1}dy\:f\left(y\right)\left[\sum_{n=0}^{\infty}\frac{2n+1}{2}P_{n}\left(y\right)P_{n}\left(x\right)-\sum_{n=N+1}^{\infty}\frac{2n+1}{2}P_{n}\left(y\right)P_{n}\left(x\right)\right],\\
 & \quad=c_{m}^{f},
\end{align*}
where we used the completeness property 
\[
\sum_{n=0}^{\infty}\frac{2n+1}{2}P_{n}\left(y\right)P_{n}\left(x\right)=\delta\left(x-y\right)
\]
 for the left term in the square brackets and the property (\ref{eq:LegPolyOrthog})
for the right one (we assume $f$ such that the order of integrals
and summation can be changed and the integration over $dx$ performed
first). In the expansion (\ref{eq:mometTriangularLegendre}) the coefficients
can be directly related to the (known) polynomial coefficients $\gamma_{j}^{n}$
of $P_{n}$ and moments $c_{j}^{f}$ of $f$
\begin{align}
\beta_{n} & =\frac{2n+1}{2}\int_{-1}^{1}f\left(x\right)P_{n}\left(x\right)dx=\frac{2n+1}{2}\int_{-1}^{1}f\left(x\right)\sum_{j=0}^{n}\gamma_{j}^{n}x^{j}dx\label{eq:betaCoefForLegende}\\
 & =\frac{2n+1}{2}\sum_{j=0}^{n}\gamma_{j}^{n}\int_{-1}^{1}f\left(x\right)x^{j}dx=\frac{2n+1}{2}\sum_{j=0}^{n}\gamma_{j}^{n}c_{j}^{f}\nonumber 
\end{align}
with
\begin{equation}
\gamma_{j}^{n}=2^{n}\binom{n}{j}\binom{\frac{n+j-1}{2}}{n}.\label{eq:gammaPolyCoefs}
\end{equation}
Furthermore, by choosing $\left(\overrightarrow{c_{m}^{f}}\right)_{j\leq N}=\delta_{m,j}$,
one can construct partial delta functions 
\begin{equation}
\Delta_{m}^{[N],M}=\sum_{n=0}^{N}\beta_{n}\left(\overrightarrow{c_{m}^{f}}\right)P_{n}\left(x\right).\label{eq:partialMomentConstructionLegendre}
\end{equation}

Numerical computations suggest (Fig. \ref{Fig:legOut}) that the Legendre-Fourier
series of many common functions converge to the function also outside
the $(-1,1)$ interval, which would be a distinctive feature from
the standard Fourier series. This observations needs to be supported
by rigorous arguments. 
\begin{figure}
\begin{centering}
\includegraphics[width=0.45\linewidth]{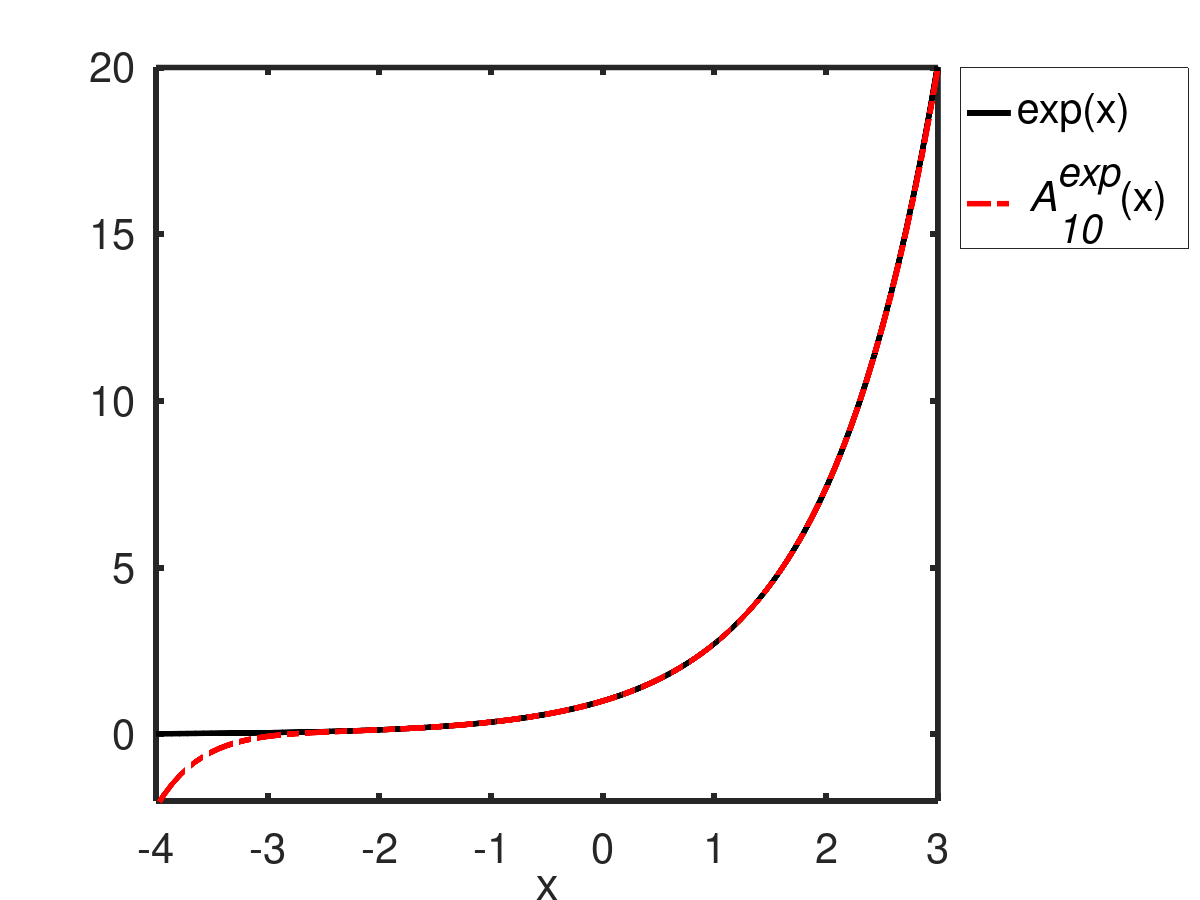} \includegraphics[width=0.45\linewidth]{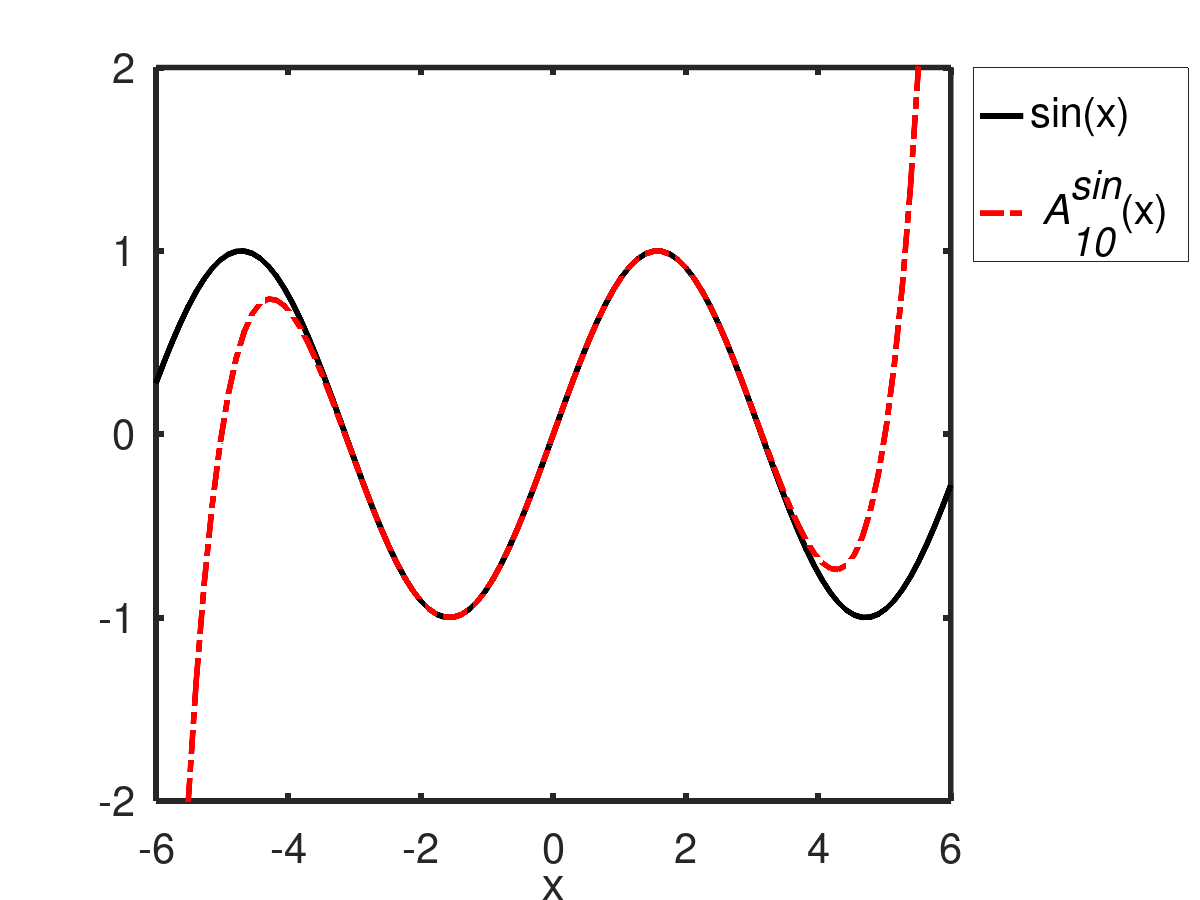}
\par\end{centering}
\caption{Examples of Legendre-Fourier expansions interpreted as triangular
moment-matching approximations shown on an interval larger than $\left(-1,1\right)$.}
\label{Fig:legOut}

\end{figure}

The notion of the moment can be generalized, the generalization is
often constructed as an integral where the power of $x$ is replaced
by an orthogonal polynomial (for an overview see \cite{10.1145/3331167}).
The importance of the moment expansion is derived from the importance
of the moments themselves: they are widely used in probability theory,
statistics, physics and many other fields too.

\subsubsection*{Higher integrals}

With the Taylor series being one of the most popular expansions, one
can ask whether a similar approximation based on higher-order integrals
(and not higher-order derivatives) could be constructed. However,
because of the integration constant freedom, the value of an anti-derivative
is not fixed. One can overcome this by setting its value in some arbitrary
way. For our purposes we define
\begin{align}
\mathcal{C}_{n}\left(f\right)\equiv c_{n}^{f}=f^{(-n)}(1),\quad\frac{d}{dx}f^{(-n)}(x) & =f^{(-n+1)}(x),\quad f^{(-n)}(-1)=0,\label{eq:hoIntegralsDef}
\end{align}
\[
\text{ for all }n>0.
\]
Even though new at the first sight (and maybe appropriate rather for
Sec. \ref{sec:newExamples}), this approximation can be related to
the moment-matching problem by the means of the Cauchy repeated-integral
formula. One has
\[
f^{(-n)}(1)\equiv c_{n}^{f}=\frac{1}{\left(n-1\right)!}\int_{-1}^{1}\left(1-t\right)^{n-1}f\left(t\right)dt.
\]
By subsequent substitutions $z=1-t$, $g\left(z\right)=f\left(1-z\right)$,
$w=z/2$ and $h(w)=g(2w)$ we arrive to
\[
c_{n}^{f}=\frac{1}{\left(n-1\right)!}\int_{0}^{2}z^{n-1}g\left(z\right)dz=\frac{2^{n}}{\left(n-1\right)!}\int_{0}^{1}w^{n-1}h\left(w\right)dw\equiv\frac{2^{n}}{\left(n-1\right)!}m_{n-1}^{h}.
\]
One can define $m_{n}^{h}=n!c_{n+1}^{f}/2^{n+1}$ ($n=0,1,\ldots$)
and interpret the latter as moments computed on the interval $\left(0,1\right)$.
The construction of all objects (approximations, delta functions)
on this interval is fully analogous to $\left(-1,1\right)$, only
the set of orthogonal polynomials is now represented by the shifted
Legendre polynomials $P_{n}^{\left(0,1\right)}$ and expressions (\ref{eq:betaCoefForLegende})
and (\ref{eq:gammaPolyCoefs}) are modified
\[
\widetilde{\beta}_{n}=\left(2n+1\right)\int_{0}^{1}f\left(x\right)P_{n}^{\left(0,1\right)}dx,\quad\widetilde{\gamma}_{j}^{n}=\left(-1\right)^{n+j}\binom{n}{j}\binom{n+j}{j}.
\]
By building a moment-matching approximation $\mathcal{A}^{h,M}\left(x\right)$
of $h$, a higher-integral approximation $\mathcal{A}^{f,I}$ of $f$
can be constructed
\[
\mathcal{A}^{f,I}\left(x\right)=\mathcal{A}^{h,M}\left(\frac{1-x}{2}\right).
\]
One may notice, that, unlike for the Taylor polynomials, the rules
(\ref{eq:hoIntegralsDef}) do not fix the function value at $x=1$
(are not applied at the zeroth order $n=0$). Indeed, the moments
$\left\{ m_{n}^{h}\right\} _{n=0}^{\infty}$ fully determine the function
$h$ (not only its shape but also its normalization) and thus $f$.

\subsection{Matching integrals of higher-order derivatives}

An approximation can be constructed upon functionals
\[
\mathcal{C}_{n}=\int_{a}^{b}dx\,\frac{d^{n}}{dx^{n}},\;c_{n}^{f}=\int_{a}^{b}f^{\left(n\right)}\left(x\right)\,dx\equiv f^{\left(n-1\right)}\left(b\right)-f^{\left(n-1\right)}\left(a\right).
\]

\subsubsection*{Bernoulli polynomial series}

Choosing this time $a=0$ and $b=1$, one can find an approximation
$\mathcal{A}^{f,B}$ constructed as a series with multiplicative coefficients
based on delta functions
\[
\mathcal{A}^{f,B}(x)=\sum_{n=0}^{\infty}a_{n}\Delta_{n}(x),\quad\Delta_{n}(x)=\frac{1}{n!}B_{n}\left(x\right),
\]
where $B_{n}$ are the Bernoulli polynomials. The matching is done
by setting
\[
a_{n}=c_{n}^{f}.
\]
A peculiar situation happens for $n=0$, where one needs to compute
an integral. However the $c_{0}$ coefficient corresponds only to
a global shift (up/down) of $f$, because $B_{0}\left(x\right)$ is
just a constant. Thus the integration can be avoided: one can build
the expansion $\mathcal{A}_{n>0}^{f,B}(x)$ for all $n>0$ and then
adjust the normalization by matching the value at some point $0\leq x_{0}\leq1$
\[
\mathcal{A}^{f,B}(x)=\mathcal{A}_{n>0}^{f,B}(x)+f(x_{0})-\mathcal{A}_{n>0}^{f,B}(x_{0}).
\]
Ignoring the latter complication, the expansion has a lot of beauty
since it is almost as easy to build as the Taylor series, one only
needs to know the derivatives at two points instead of one. Despite
the simplicity, it seems not to be known very well, although appearing
in the literature at several places \cite{krylov1962,jordan1965,MORDELL1966132}.
One can consult the latter reference to address the convergence questions.
Some problems regarding the convergence can be easily seen when realizing
how the delta functions are constructed. One can formally introduce
a function $\Delta_{-1}$ satisfying
\begin{equation}
\mathcal{C}_{n}\left(\Delta_{-1}\right)=0\text{ for all }n\geq0.\label{eq:deltaForBernoulli}
\end{equation}
By integration one defines functions $\Delta_{n\geq0}$, $\Delta'_{n+1}=\Delta{}_{n}$
and by a careful choice of integration constants one can fulfill the
delta property $\mathcal{C}_{n}\left(\Delta_{m}\right)=\delta_{n,m}$.
A natural choice $\Delta_{-1}(x)\equiv0$ leads to the Bernoulli polynomials.
However there are other functions, such as for example $\cos\left(2\pi x-\pi\right)$,
which also fulfill (\ref{eq:deltaForBernoulli}) and which would lead,
performing the integration, to different delta functions. Thus one
cannot expect the Bernoulli-polynomials based expansion to converge
for $\cos\left(2\pi x-\pi\right)$ on $\left(0,1\right)$.

The approximation can be scaled to an arbitrary finite interval $\left(a,b\right)$
by scaling the argument
\[
\mathcal{A}_{\left(a,b\right)}^{f,B}(x)=\sum_{n=0}^{\infty}a_{n}\Delta_{n}^{\left(a,b\right)}(x),\quad\Delta_{n}^{\left(a,b\right)}(x)=\left(b-a\right)^{n-1}\Delta_{n}(\frac{x-a}{b-a}).
\]
A natural question rises about the behavior of $\mathcal{A}_{\left(a,b\right)}^{f,B}$
for $b\rightarrow a$.
\begin{prop}
For an analytic function $f$, the partial approximation $\mathcal{A}_{\left(a,b\right),N}^{f,B}$
becomes in the limit $b\rightarrow a$ its Taylor polynomial.
\end{prop}
\begin{proof}
Using the notation $b=a+\varepsilon$, straightforward computations
yield
\[
\mathcal{A}_{\left(a,a+\varepsilon\right),N}^{f,B}=\sum_{n=0}^{N}\left[f^{\left(n-1\right)}\left(a+\varepsilon\right)-f^{\left(n-1\right)}\left(a\right)\right]\varepsilon^{n-1}\frac{1}{n!}B_{n}\left(\frac{x-a}{\varepsilon}\right).
\]
Writing $f^{\left(n-1\right)}\left(a+\varepsilon\right)=f^{\left(n-1\right)}\left(a\right)+\varepsilon f^{\left(n\right)}\left(a\right)+\mathcal{O}\left(\varepsilon^{2}\right)$
and using the notation $f^{\left(n\right)}\equiv f^{\left(n\right)}\left(a\right)$
one has
\[
\mathcal{A}_{\left(a,a+\varepsilon\right),N}^{f,B}=\sum_{n=0}^{N}\left[\varepsilon f^{\left(n\right)}+\mathcal{O}\left(\varepsilon^{2}\right)\right]\varepsilon^{n-1}\frac{1}{n!}\sum_{k=0}^{n}\frac{n!}{(n-k)!k!}B_{n-k}\varepsilon^{-k}\left(x-a\right)^{k},
\]
where explicit formulas for the Bernoulli polynomials were used with
$B_{n-k}$ denoting the Bernoulli numbers. Re-arranging the sums one
arrives to
\begin{equation}
\mathcal{A}_{\left(a,a+\varepsilon\right),N}^{f,B}=\sum_{k=0}^{N}\frac{1}{k!}\left\{ \sum_{n=k}^{N}\varepsilon^{n-k}\frac{f^{\left(n\right)}+\mathcal{O}\left(\varepsilon\right)}{(n-k)!}B_{n-k}\right\} \left(x-a\right)^{k},\label{eq:sumsReArangement}
\end{equation}
One can study an individual term in front of the powers of $\left(x-a\right)$
in the limit $\varepsilon\rightarrow0$. In this limit the dominant
$\varepsilon$ term is the one where $n=k$. Thus we have
\[
\lim_{\varepsilon\rightarrow0}\sum_{n=k}^{N}\varepsilon^{n-k}\frac{f^{\left(n\right)}+\mathcal{O}\left(\varepsilon\right)}{(n-k)!}B_{n-k}\overset{n=k}{=}f^{\left(k\right)}
\]
leading to
\[
\lim_{\varepsilon\rightarrow0}\mathcal{A}_{\left(a,a+\varepsilon\right),N}^{f,B}=\sum_{k=0}^{N}\frac{1}{k!}f^{\left(k\right)}\left(x-a\right)^{k}
\]
which corresponds to the Taylor polynomial.
\end{proof}
Assuming the re-arrangement of the sums (\ref{eq:sumsReArangement})
is justified for $N=\infty$, the whole proof is valid for a full
approximation and the Taylor series. Also, the analyticity condition
can be presumably relaxed if the error in the small parameter expansion
of $f^{\left(n-1\right)}\left(\varepsilon\right)$ is treated carefully.
Numerical observations indicate a rich set of interesting properties
which await to be addressed rigorously:
\begin{itemize}
\item Functions such as $\cos\left(x\right)$ do not have (as expected)
a convergent approximation $\mathcal{A}_{\left(a,b\right)}^{f,B}$
for $b-a=2\pi$. However, if the length of the interval differs only
slightly from (a multiple of) $2\pi$, the approximation seems to
converge.
\item Numerical computations suggest that for many common functions the
approximation $\mathcal{A}_{\left(a,b\right)}^{f,B}$ fails to converge
once the interval length goes over some limit. For $\exp\left(x\right)$
approximated on a symmetric interval $\left(-a,a\right)$ the convergence
seems to be lost for $a\apprge\pi$.
\item The quality of partial approximations $\mathcal{A}_{N}^{f,B}$may
increase with increasing $N$ up to some $N_{0}$, and decrease (become
divergent) afterwards. Such behavior is observed for $f(x)=\sqrt{4-x^{2}}$
approximated on $\left(-1,1\right)$ where the best-quality partial
approximation $\mathcal{A}_{\left(-1,1\right),N}^{f,B}$ is realized
for $N=4$ and $N=5$.
\end{itemize}

\subsection{Value matching}

Approximations (often partial) meant to match function values 
\[
\mathcal{C}_{n}\left(f\right)=f\left(x_{n}\right),\;c_{n}^{f}=f\left(x_{n}\right),\;\left\{ x_{n}\right\} _{n=0}^{N,\infty},\;x_{n}\epsilon\left(a,b\right),
\]
where $\left\{ x_{n}\right\} $ is a predefined set of numbers, are
usually referred to as interpolations and represent a very large topic
covered by many resources. Hence, we present only the most popular
ones, among them the Lagrange form of the interpolation polynomial.
It is a delta approximation
\[
\mathcal{A}_{N}^{f,V}(x)=\sum_{n=0}^{N}a_{n}\Delta_{n}(x),\quad\Delta_{n}(x)=\prod_{i=0,i\neq n}^{N}\frac{(x-x_{i})}{(x_{n}-x_{i})},\;a_{n}=c_{n}^{f},
\]
where the delta property is reached by the creation of zeros in the
numerator for all $\left\{ x_{i}\right\} $ except $x_{n}$. The denominator
then guarantees a correct normalization $\Delta_{n}(x_{n})=1$. The
Newton form of the interpolation polynomial is a modification of the
previous and provides a triangular approximation
\[
\mathcal{A}_{N}^{f,V'}(x)=\sum_{n=0}^{N}a_{n}\nabla_{n}(x),\quad\nabla_{n}(x)=\prod_{i=0}^{n-1}(x-x_{i}),\;a_{n}=a_{n}\left(c_{0}^{f},\ldots,c_{n}^{f}\right)
\]
for which the individual terms of the sum are defined also in the
limit $N\rightarrow\infty$, the coefficients $a_{n}$ are divided
differences. The idea of creating zeros can be generalized and an
approximation built using an arbitrary function $\rho$ such that
$\rho\left(0\right)=0$
\[
\mathcal{A}_{N}^{f,V''}\left(x\right)=\sum_{n=0}^{N}c_{n}^{f}\prod_{k=0,k\neq n}^{N}\frac{\rho\left(x-x_{k}\right)}{\rho\left(x_{n}-x_{k}\right)}.
\]
For $\rho\left(x\right)=\sin\left(x\right)$ this leads to a trigonometric
interpolation (of a non-minimal degree) and Fourier series.

Several other approximations can be found on the market, usually less
popular (e.g. the Whittaker--Shannon interpolation formula, which
we address in Sec. \ref{subsec:Variations-on-Whittaker=002013Shannon}
or \cite{Filomat}). Nevertheless a number of theoretical results
have been derived in this domain, most of them related to the existence
of an entire function (in the complex plane) with the interpolation
property. One can mention e.g. the Nevanlinna--Pick problem \cite{zbMATH02616893}
or the Weierstrass factorization theorem (which turns $f$ into an
infinite product).

This kind of approximations can be useful for functions which are
easy to evaluate on some countable set of their arguments, thus providing
a way to approximate them elsewhere. The structure of this set is
usually specified by the method itself (e.g. equidistant points for
the Whittaker--Shannon formula) and cannot be later changed to suit
the function.

\section{New examples\label{sec:newExamples}}

\subsection{Matching derivatives}

In this sub-section we present new approximations based on the derivative
matching
\[
c_{n}^{f}=\frac{d^{n}f\left(x\right)}{dx^{n}}|_{x=x_{0}=0}.
\]
We propose three expansions exploiting a similar idea which is to
build a series with multiplicative coefficients where, by construction,
the coefficients enter into the game progressively when higher order
derivatives are performed (i.e. the approximations are triangular).
The fourth example is somewhat different, it deals with a partial
delta approximation.

\subsubsection{Expansion into $g\left(x\right)\sum a_{n}x^{n},\;g\left(0\right)\protect\neq0$}

An approximation of the form
\begin{equation}
\mathcal{A}^{f}\left(x\right)=g\left(x\right)\sum_{n=0}^{\infty}a_{n}x^{n},\;g\left(0\right)\neq0\label{eq:gFaktorVpredu}
\end{equation}
is triangular because, following the product and chain differentiation
rules, the term $g\left(x\right)x^{n}$ becomes non-zero only after
$n$ differentiations. To make a connection between power-expansion
coefficients of $f$, $f_{n}=c_{n}^{f}/n!$ and those of $g$, $g_{n}=g^{\left(n\right)}\left(0\right)/n!$
let us assume that both functions are analytic. Then, the Cauchy product
of $\sum a_{k}x^{k}$ and $g\left(x\right)=\sum g_{k}x^{k}$ on the
right-hand side of (\ref{eq:gFaktorVpredu}) yields
\[
\sum_{k=0}^{\infty}f_{k}x^{k}\overset{!}{=}\sum_{k=0}^{\infty}\left[\sum_{n=0}^{k}a_{n}g_{k-n}\right]x^{k},
\]
where the expression in the square brackets corresponds to a lower
triangular Toeplitz matrix
\[
\left(\begin{array}{c}
f_{0}\\
f_{1}\\
f_{2}\\
\vdots
\end{array}\right)=\left(\begin{array}{cccc}
g_{0} & 0 & 0 & \cdots\\
g_{1} & g_{0} & 0 & \cdots\\
g_{2} & g_{1} & g_{0} & \cdots\\
\vdots & \vdots & \vdots & \ddots
\end{array}\right)\left(\begin{array}{c}
a_{0}\\
a_{1}\\
a_{2}\\
\vdots
\end{array}\right).
\]
To invert the matrix a recursive formula based on generalized Fibonacci
polynomials is available \cite{SahinAdem}. The $a_{n}$ coefficients
in (\ref{eq:gFaktorVpredu}) are expressed in terms of $g_{n}$ and
$f_{n}$ as follows 
\[
a_{n}=\frac{1}{g_{0}}\sum_{i=0}^{n}\Lambda_{n-i}f_{i},\text{ where }\varLambda_{k}=\begin{cases}
1 & \text{ if }k=0\\
-\left(\sum_{j=1}^{k}g_{j}\varLambda_{k-j}\right)/g_{0} & \text{else}
\end{cases}.
\]
The expansion can be shifted to an arbitrary point $x_{0}\epsilon\mathbb{R}$
by shifting its argument. One also notes that the expansion is equivalent
to building the Taylor series for $h(x)\equiv f(x)/g\left(x\right)$,
its novelty lies in finding the relation between the coefficients
$\left\{ a_{i}\right\} $ and the derivatives of $f$ (an not those
of $h$). In order to provide an example without recurrent relations
we chose $g\left(x\right)=\exp\left(wx^{q}\right)$ and prove
\begin{prop}
A smooth function $f$can be approximated at $x_{0}=0$ by
\begin{equation}
\mathcal{A}^{f,D_{1}}\left(x\right)=\exp\left(wx^{q}\right)\sum_{n=0}^{\infty}a_{n}x^{n},\quad q\epsilon\mathbb{N^{+}},w\epsilon\mathbb{R},\label{eq:newExpansion_D1}
\end{equation}
with
\begin{equation}
a_{n}=\sum_{i=0}^{n}m_{n,i}c_{i}^{f},\quad m_{n,i}=\delta_{v_{n-i},0}\frac{\left(-1\right)^{u_{n-i}}w^{\left\llbracket u_{n-i}\right\rrbracket }}{\left(v_{n}+qu_{i}\right)!u_{n-i}!},\label{eq:newExp_1_coef}
\end{equation}
\[
u_{n}=\left\lfloor \frac{n}{q}\right\rfloor ,\quad v_{n}=n\%q\text{ and }x^{\left\llbracket y\right\rrbracket }=\begin{cases}
1 & \text{\text{for }}x=y=0\\
x^{y} & \text{else (when defined)}
\end{cases},
\]
where $\left\lfloor \:\right\rfloor $ is the floor function, $\%$
represents the modulo operation and the generalized exponentiation
allows (otherwise undefined) expansion for cases $w=0$.
\end{prop}
\begin{proof}
Using the general Leibniz rule we differentiate (\ref{eq:newExpansion_D1})
$m$ times
\[
\frac{d^{m}}{dx^{m}}\left[\exp\left(wx^{q}\right)\sum_{n=0}^{\infty}a_{n}x^{n}\right]=\sum_{k=0}^{m}\binom{m}{k}\left[\sum_{n=0}^{\infty}a_{n}x^{n}\right]^{(m-k)}\left[\exp\left(wx^{q}\right)\right]^{(k)}.
\]
The first factor becomes 
\[
\left[\sum_{n=0}^{\infty}a_{n}x^{n}\right]_{x=0}^{(m-k)}=\left[\sum_{n=0}^{\infty}\frac{(n+m-k)!}{n!}a_{n+m-k}x^{n}\right]_{x=0}=(m-k)!a_{m-k},
\]
where we are interested in the absolute term which determines the
derivative at zero. The second factor can be deduced from the repeated
differentiation of the corresponding power expansion. One has
\[
\left[\frac{d^{k}}{dx^{k}}\exp\left(wx^{q}\right)\right]_{x=0}=\left[\frac{d^{k}}{dx^{k}}\sum_{n=0}^{\infty}\frac{1}{n!}\left(wx^{q}\right)^{n}\right]_{x=0}=\delta_{v_{k},0}w^{u_{k}}\frac{(u_{k}q)!}{u_{k}!}.
\]
Combining the two results one arrives at
\begin{align}
\frac{d^{m}}{dx^{m}}\left[\exp\left(wx^{q}\right)\sum_{n=0}^{\infty}a_{n}x^{n}\right]_{x=0} & =\sum_{j=0}^{m}\delta_{v_{m-j},0}\frac{m!(u_{m-j}q)!}{(m-j)!u_{m-j}!}w^{u_{m-j}}a_{j},\label{eq:maticaD}\\
 & \equiv\sum_{j=0}^{m}D_{m,j}a_{j},\nonumber 
\end{align}
which gives derivatives as a function of the coefficients $a_{j}$.
Using the matrix language (and defining $D_{m,j}=0$ for $j>m$),
the dependence of coefficients on derivatives is provided by $D^{-1}$.
One has
\begin{equation}
\left(D^{-1}\right)_{n,i\leq n}=\delta_{v_{n-i},0}\frac{\left(-1\right)^{u_{n-i}}w^{\left\llbracket u_{n-i}\right\rrbracket }}{\left(v_{n}+qu_{i}\right)!u_{n-i}!}\label{maticaInvD}
\end{equation}
with $\left(D^{-1}\right)_{n,i>n}=0$. The (infinite) matrices $D$
and $D^{-1}$ are lower triangular, thus $\left(DD^{-1}\right)_{n,i>n}=0$.
For $n=i$ the computation is straightforward ($u_{0}=0$, $v_{0}=0$)
\begin{align*}
\left(DD^{-1}\right)_{i,i} & =\sum_{j=0}^{\infty}D_{i,j}\left(D^{-1}\right)_{j,i}=D_{i,i}\left(D^{-1}\right)_{i,i}=i!\times\frac{1}{\left(v_{i}+qu_{i}\right)!}=1.
\end{align*}
A non-trivial computation arises for
\begin{align}
\left(DD^{-1}\right)_{m,i<m} & =\sum_{j=0}^{\infty}D_{m,j}\left(D^{-1}\right)_{j,i}=\sum_{j=i}^{m}D_{m,j}\left(D^{-1}\right)_{j,i}\label{eq:checkInversity}
\end{align}
The delta functions in (\ref{eq:maticaD}) (\ref{maticaInvD}) imply
$D$ and $D^{-1}$are block matrices with $q\times q$ sub-matrices
which are diagonal. The $\left\{ c_{k}^{f}\right\} $ and $\left\{ a_{k}\right\} $
sets so split into $q$ subsets which transform independently and
are characterized by the value of $w=0,\ldots,q-1$. One can write
\[
\left(D_{w}^{\left(-1\right)}\right)_{p,k}\equiv\left(D^{\left(-1\right)}\right)_{w+pq,w+kq}
\]
and analyze for $k<p$
\begin{align*}
\left(D_{w}D_{w}^{-1}\right)_{p,k} & =\sum_{t=k}^{p}\left(D_{w}\right)_{p,t}\left(D_{w}^{-1}\right)_{t,k}\\
 & =\sum_{t=k}^{p}\frac{\left(w+pq\right)!(u_{w+pq-w-tq}q)!}{(w+pq-w-tq)!u_{w+pq-w-tq}!}w^{u_{w+pq-w-tq}}\\
 & \qquad\qquad\times\frac{\left(-1\right)^{u_{w+tq-w-kq}}w^{\left\llbracket u_{w+tq-w-kq}\right\rrbracket }}{\left(v_{w+tq}+qu_{w+kq}\right)!u_{w+tq-w-kq}!}
\end{align*}
One has
\begin{align*}
u_{pq-tq} & =p-t,\quad u_{tq-kq}=t-k,\quad v_{w+tq}=v_{w},\quad u_{w+kq}=u_{w}+k
\end{align*}
which leads to
\begin{align*}
\left(D_{w}D_{w}^{-1}\right)_{p,k} & =\sum_{t=k}^{p}\frac{\left(w+pq\right)!}{\left(p-t\right)!}w^{p-t}\times\frac{\left(-1\right)^{t-k}w^{\left\llbracket t-k\right\rrbracket }}{\left(w+qk\right)!\left(t-k\right)!}\\
 & =w^{p-k}\frac{\left(w+pq\right)!}{\left(w+qk\right)!}\sum_{t=k}^{p}\frac{\left(-1\right)^{t-k}}{\left(p-t\right)!\left(t-k\right)!}.
\end{align*}
Using a substitution $s=t-k$ with the subsequent definition $r=p-k$
the latter sum becomes
\[
\sum_{t=k}^{p}\frac{\left(-1\right)^{t-k}}{\left(p-t\right)!\left(t-k\right)!}=\sum_{s=0}^{r}\frac{\left(-1\right)^{s}}{\left(r-s\right)!s!}=\frac{1}{r!}\sum_{s=0}^{r}\left(-1\right)^{s}\binom{r}{s}=0.
\]
The proof uses the standard exponentiation, the validity of the generalized
exponentiation can be for $w=0$ easily checked.
\end{proof}
The special case $w=0$ corresponds to the Taylor series. The convergence
properties can be derived from the criteria used for the Taylor series
applied to the $\left\{ a_{i}\right\} $ coefficients. Indeed, the
exponential is an entire function in the whole complex plane and thus
its power series converges absolutely at each complex point. Therefore,
for a given $x$, the expression (\ref{eq:newExpansion_D1}) can be
interpreted as a multiplication of two sequences, one of which is
absolutely convergent. Then, by the Mertens' convergence theorem,
the convergence of $\sum_{n=0}^{\infty}a_{n}x^{n}$ implies the convergence
of the whole expression which converges to the product of the two
series.

The proposed expansion may be interesting for several reasons. If
one chooses, for example, an even and negative $w$, then each partial
approximation approaches zero as $\left|x\right|\rightarrow\infty$
and has a finite integral over $\mathbb{R}$ (unlike e.g. the Taylor
polynomials). Also, the structure as in (\ref{eq:newExpansion_D1})
is very often to be seen in physics and mathematics, e.g. the radial
parts of wave functions for commonly studied spherical potentials
take in many cases this form. Also, if an extension to complex $q$
is possible (not studied here), the form of the approximation fits
the definition of spherical harmonics. Example approximations using
(\ref{eq:newExpansion_D1}) are shown in Fig \ref{fig:expPoly}.
\begin{figure}
\begin{centering}
\includegraphics[width=0.45\textwidth]{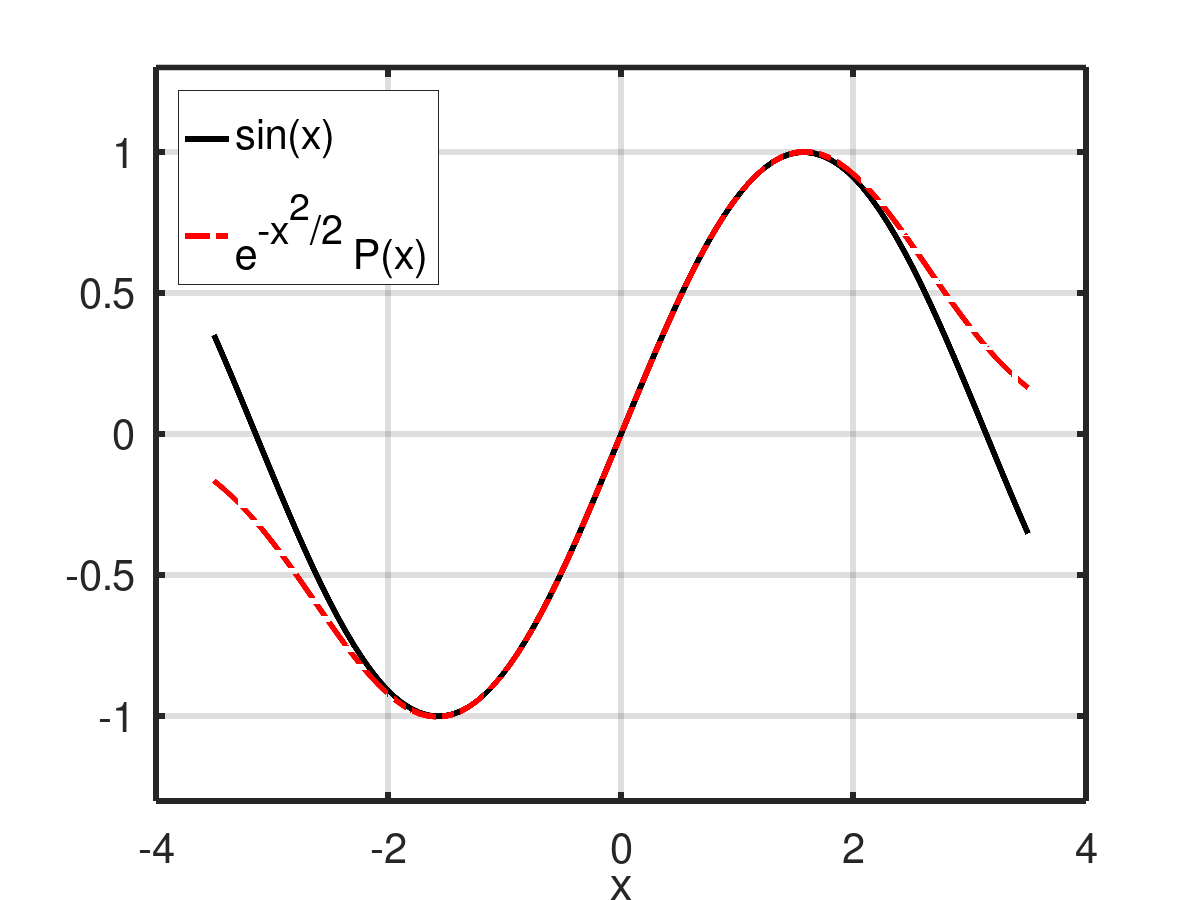} \includegraphics[width=0.45\textwidth]{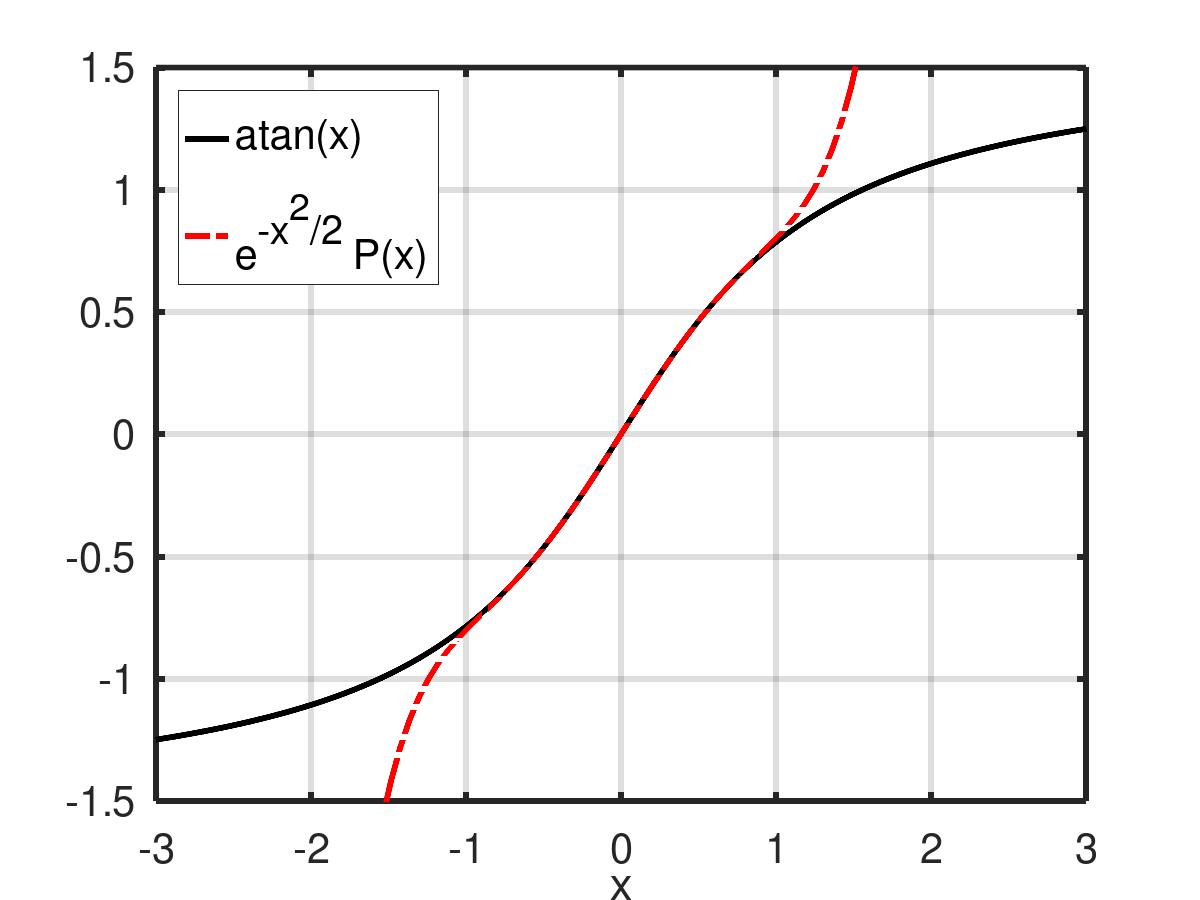}
\par\end{centering}
\caption{Approximating $\sin\left(x\right)$ and $\arctan\left(x\right)$ by
expression (\ref{eq:newExpansion_D1}) with 11 terms, $q=2$ and $w=-1/2$.\label{fig:expPoly}}

\end{figure}

\subsubsection{Expansion into $\sum a_{n}\left[g\left(x\right)\right]^{n},\;g\left(0\right)=0,\;g'\left(0\right)\protect\neq0$\label{subsec:Power-series-of}}

The next derivative-matching expansion we propose
\begin{equation}
\mathcal{A}^{f,D_{2}}\left(x\right)=\sum_{n=0}^{\infty}a_{n}\left[g\left(x\right)\right]^{n}\text{ with }\;g\left(0\right)=0\text{ and }g'\left(0\right)\neq0\label{eq:xLikeExpansion}
\end{equation}
uses a function $g$ which is for small $x\ll1$ similar to $\alpha x$,
$\alpha\neq0$ (and thus invertible on some neighborhood of zero).
It is a triangular approximation because, following the product and
chain differentiation rules, the term $\left[g\left(x\right)\right]^{n}$
becomes non-zero only after $n$ differentiations. We encountered
an example of this type in Sec. \ref{subsec:Powers-of-sines} which
represents a special case with $g\left(x\right)=2\sin\left(x/2\right)$.
Assuming the analyticity of $\mathcal{A}^{f,D_{2}}$ and its convergence
to $f$, $\mathcal{A}^{f,D_{2}}=f$, one can introduce the substitution\footnote{Idea of Lukáš Holka.}
\[
x=g^{-1}\left(y\right)\text{ leading to }f\left(g^{-1}\left(y\right)\right)=\sum_{n=0}^{\infty}a_{n}y^{n}.
\]
Thus $\left\{ a_{n}\right\} $ can be obtained as the Taylor series
coefficients of $f\left(g^{-1}\left(x\right)\right)$
\begin{equation}
a_{n}=\frac{1}{n!}\frac{d^{n}}{dx^{n}}f\left(g^{-1}\left(x\right)\right)|_{x=0}.\label{eq:LukasovVzorec}
\end{equation}
The complicated structure of higher-order derivatives of (\ref{eq:LukasovVzorec})
makes an approach based on a general function $g$ impractical and
one rather searches for feasible special cases.
\begin{prop}
One presumably new special case is represented by
\begin{equation}
\mathcal{A}^{f,D_{3}}\left(x\right)=\sum_{n=0}^{\infty}a_{n}\left[\ln\left(x+1\right)\right]^{n}\label{eq:powOfLogs}
\end{equation}
with 
\begin{equation}
a_{n}=\frac{1}{n!}\sum_{k=0}^{n}c_{k}^{f}b_{n,k},\quad b_{n,k}=\frac{1}{k!}\sum_{i=0}^{k}\left(-1\right)^{i}\binom{k}{i}(k-i)^{\left\llbracket n\right\rrbracket },\label{eq:logPowCoefs}
\end{equation}
where $b_{n,k}$ are the Stirling numbers of the second kind and $a^{\left\llbracket b\right\rrbracket }$
denotes the generalized exponentiation as introduced earlier.
\end{prop}
\begin{proof}
The Faà di Bruno's formula can be written in terms of the power series
of the composed function
\begin{align*}
f\left(g^{-1}\left(x\right)\right) & =c_{0}^{f}+\sum_{n=1}^{\infty}\left[\frac{1}{n!}\sum_{k=1}^{n}c_{k}^{f}B_{n,k}\left(g_{1}^{-1},\ldots,g_{n-k+1}^{-1}\right)\right]x^{n},\\
g_{i}^{-1} & \equiv\frac{d^{i}}{dx^{i}}g^{-1}\left(x\right)|_{x=0},
\end{align*}
where $B_{n,k}$ are the Bell polynomials. We have $g^{-1}\left(x\right)=\exp\left(x\right)-1$
thus 
\[
g_{i>0}^{-1}=1
\]
which leads to (\ref{eq:LukasovVzorec})
\[
a_{n}=\frac{1}{n!}\frac{d^{n}}{dx^{n}}f\left(g^{-1}\left(x\right)\right)|_{x=0}=\frac{1}{n!}\sum_{k=1}^{n}c_{k}^{f}B_{n,k}(1,1,\ldots,1)=\frac{1}{n!}\sum_{k=1}^{n}c_{k}^{f}b_{n,k}.
\]
The generalized exponentiation is used to define the $k=0$ case $b_{n,0}=\delta_{n,0}$
thus obtaining for all $n\geq0$
\[
a_{n}=\frac{1}{n!}\sum_{k=0}^{n}c_{k}^{f}b_{n,k}.
\]
\end{proof}
The existence of various formulas involving the Bell polynomials,
e.g.\footnote{Here $\left[\begin{smallmatrix}n\\
k
\end{smallmatrix}\right]$ denotes the unsigned Stirling numbers of the first kind.} 
\begin{align}
B_{n,k}(0!,1!,\ldots,\left(n-k\right)!) & =\left[\begin{array}{c}
n\\
k
\end{array}\right],\nonumber \\
B_{n,k}(1!,2!,\ldots,\left(n-k+1\right)!) & =\binom{n-1}{k-1}\frac{n!}{k!},\label{eq:threeFormulas}\\
B_{n,k}(1,2,\ldots,n-k+1) & =\binom{n}{k}k^{n-k},\nonumber 
\end{align}
implies the feasibility of the expansion (\ref{eq:xLikeExpansion})
for those function $g^{-1}$ whose higher derivatives appear as arguments
in these special cases. Choosing for example $g\left(x\right)=1-\exp\left(-x\right)$
one has $g^{-1}\left(x\right)=-\ln\left(1-x\right)$ of which the
derivatives (at zero) appear in the first of the three formulas. In
such a scenario $b_{n,k}=\left[\begin{smallmatrix}n\\
k
\end{smallmatrix}\right]$. In the second of the three formulas the factorials are just shifted,
thus leading to $g\left(x\right)=x/\left(x+1\right)$, $g^{-1}\left(x\right)=x/\left(1-x\right)$
and $b_{n,k}=\binom{n-1}{k-1}\frac{n!}{k!}$. Such an expansion
\begin{equation}
\mathcal{A}_{N}^{f,D_{4}}\left(x\right)=a_{0}+\sum_{n=1}^{N}a_{n}\left(\frac{x}{x+1}\right)^{n},\quad a_{n}=\frac{1}{n!}\sum_{k=1}^{n}\binom{n-1}{k-1}\frac{n!}{k!}c_{k}^{f},\label{eq:newPade}
\end{equation}
is actually an interesting one, because it is a rational approximation
of the function $f$. It does not fit the definition of the Padé approximant,
since it is not of a minimal degree, yet it shares many common features
with the latter and in its terminology it would be labeled as diagonal
(numerator has the same degree as the denominator). It has obvious
advantages: the derivative matching is easy, done order by order with
the persistence of coefficients and, in addition, the limit $N\rightarrow\infty$
in formula (\ref{eq:newPade}) represents a natural way of constructing
a full approximation. It has also one unpleasant feature: it has a
single singularity situated on the real axis always at $x=-1$. One
cannot produce more singularities, yet one can arbitrarily shift the
singularity by scaling the argument: For the singularity to be situated
at $x=\alpha$ one builds an approximation $\mathcal{\widetilde{A}}_{N}^{D_{4}}\left(x\right)$
with characteristic numbers $\widetilde{c}_{n}=\left(-\alpha\right)^{n}c_{n}^{f}$
which has singularity at $-1$. Then an approximation of $f$ having
the singularity at $\alpha\neq0$ can be written as $\mathcal{A}_{N}^{f,D_{4},\alpha}\left(x\right)=\mathcal{\widetilde{A}}_{N}^{D_{4}}\left(-x/\alpha\right)$.

The arguments of the Bell polynomial in the last line of (\ref{eq:threeFormulas})
correspond to the derivatives of $g^{-1}\left(x\right)=xe^{x}$ which
is invertible on $\left[-1,+\infty\right)$. The corresponding expansion
is thus build as a power series of the principal branch of the Lambert
$W$ function and is defined on $\left[-1/e,+\infty\right)$.

Several other special cases of Belle polynomial arguments are known
in the literature for which a closed formula exist \cite{QI2020124382,articleBelle2017,WANG2009104,howard1980special}.
These may be used to construct more approximations of the type (\ref{eq:xLikeExpansion}).

The convergence properties of (\ref{eq:powOfLogs}) and (\ref{eq:newPade})
remain an open question. On a computer, expressions of type (\ref{eq:xLikeExpansion})
can be evaluated using a Horner's scheme like approach to increase
efficiency. Examples of the two approximations are shown in Figs.
\ref{fig:logPowers} and \ref{fig:newPade} respectively.
\begin{figure}[t]
\begin{centering}
\includegraphics[width=0.47\linewidth]{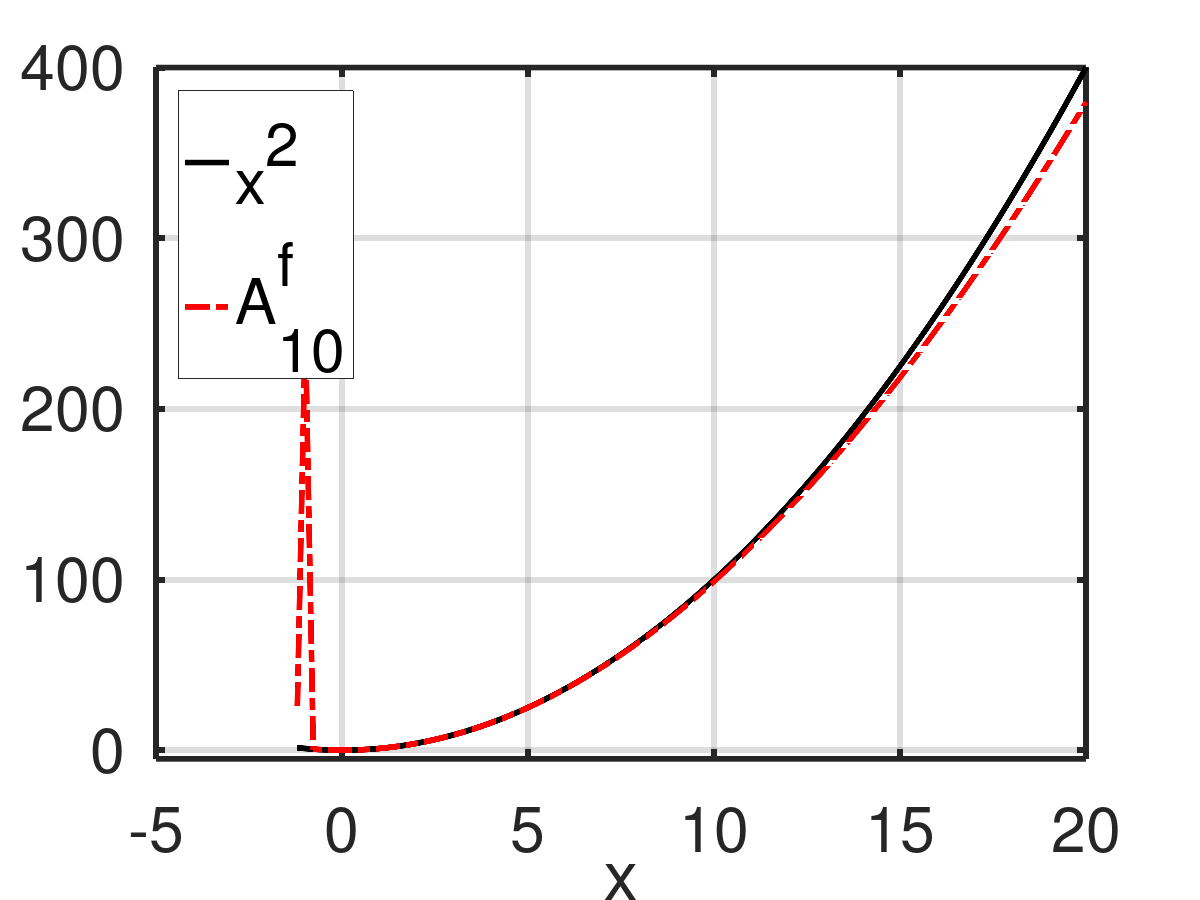} \includegraphics[width=0.47\linewidth]{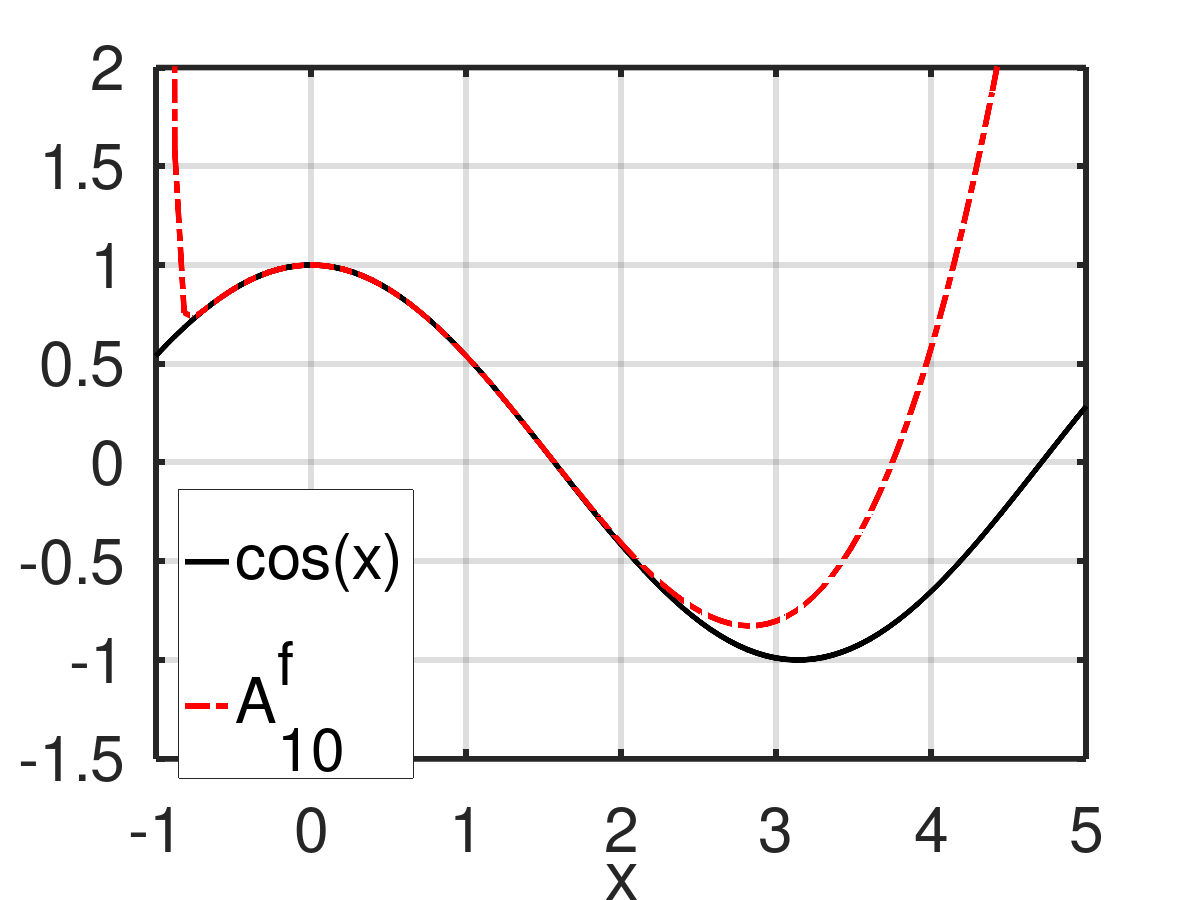}
\par\end{centering}
\begin{centering}
\includegraphics[width=0.47\linewidth]{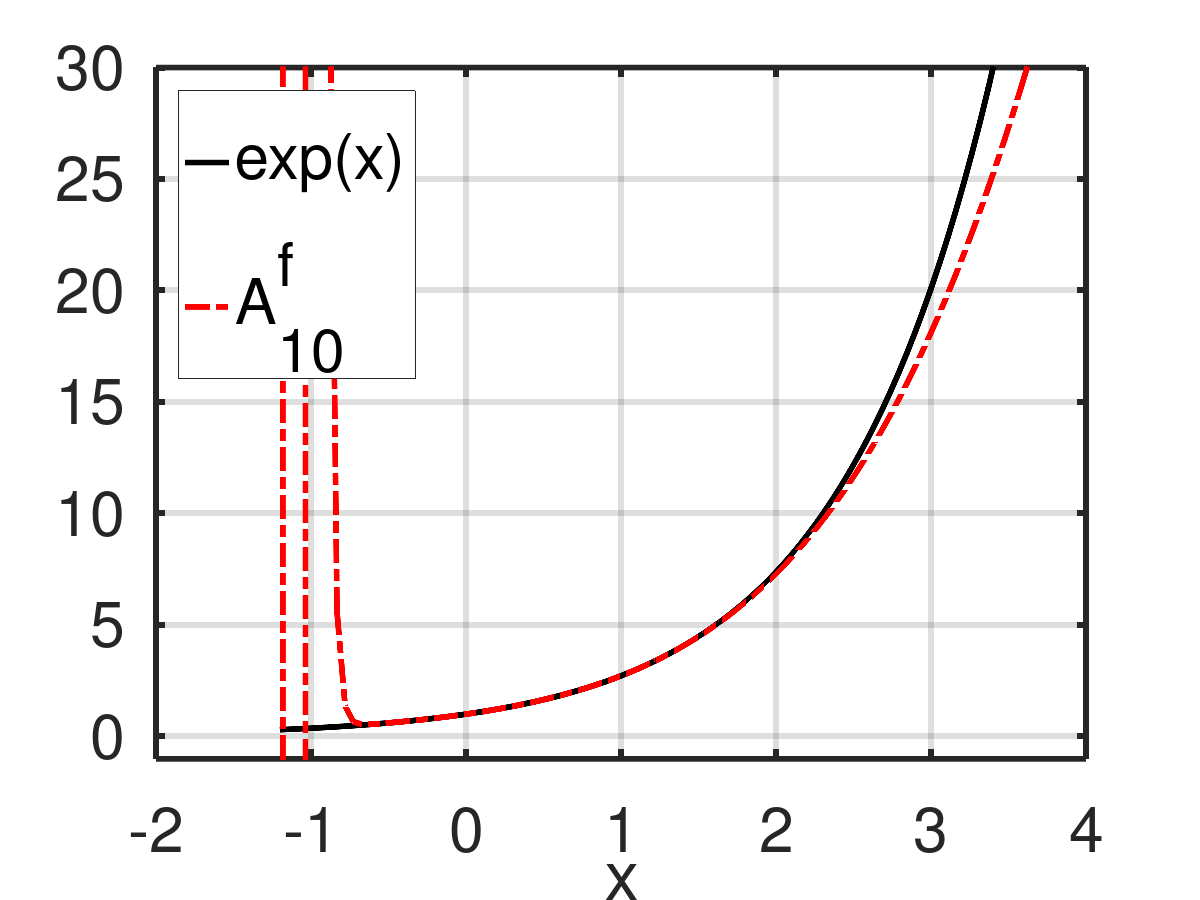} \includegraphics[width=0.47\linewidth]{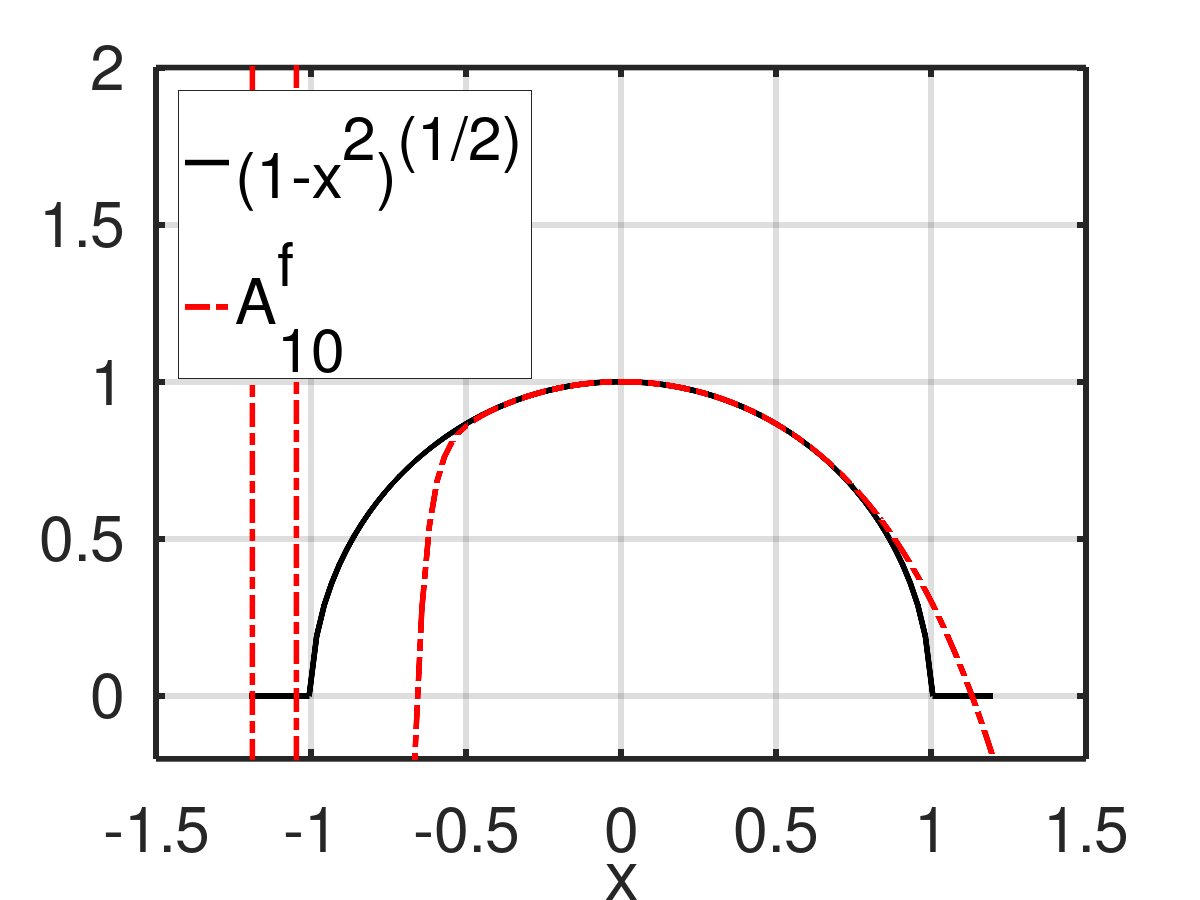}
\par\end{centering}
\caption{Illustrative examples of approximations based on powers of $g\left(x\right)=\ln\left(x+1\right)$
with 11 terms.\label{fig:logPowers}}
\end{figure}
\begin{figure}[t]
\begin{centering}
\includegraphics[width=0.47\linewidth]{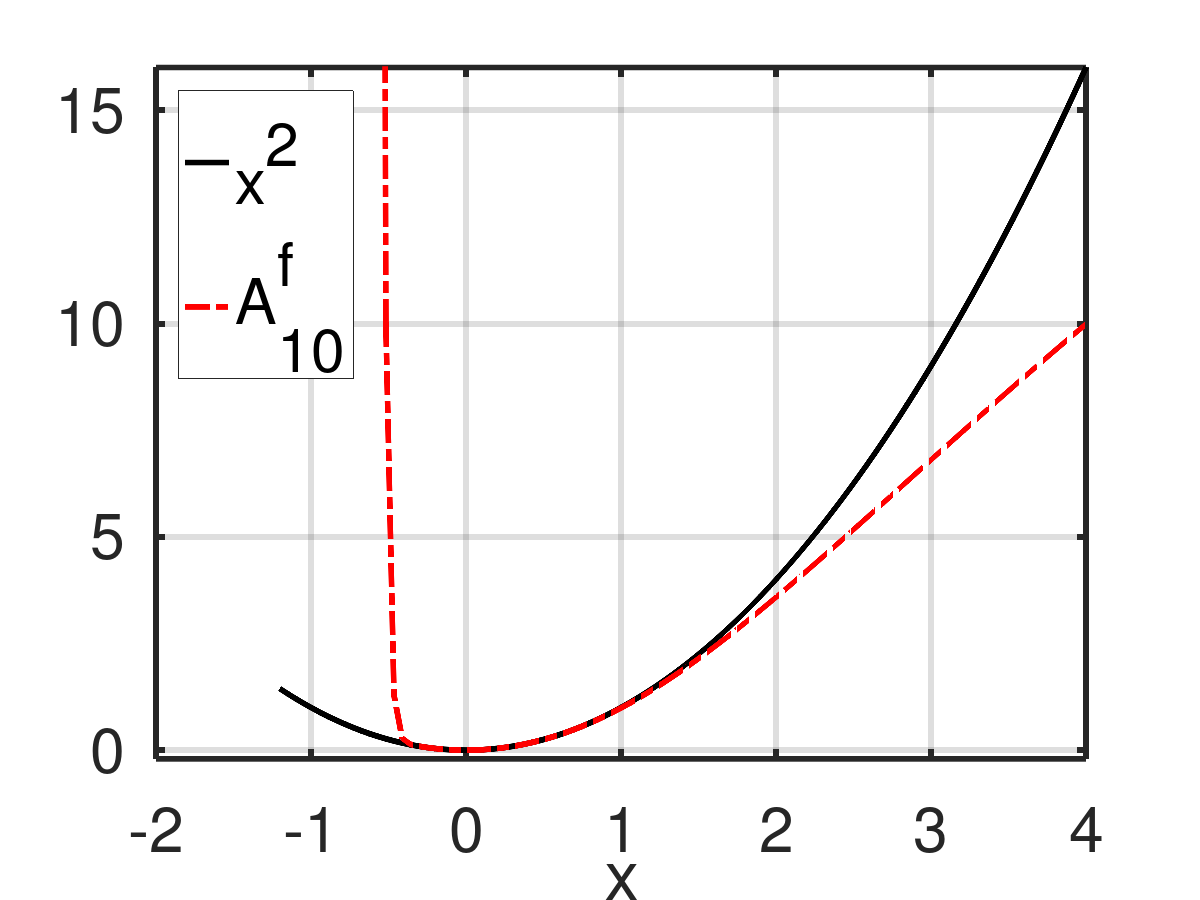} \includegraphics[width=0.47\linewidth]{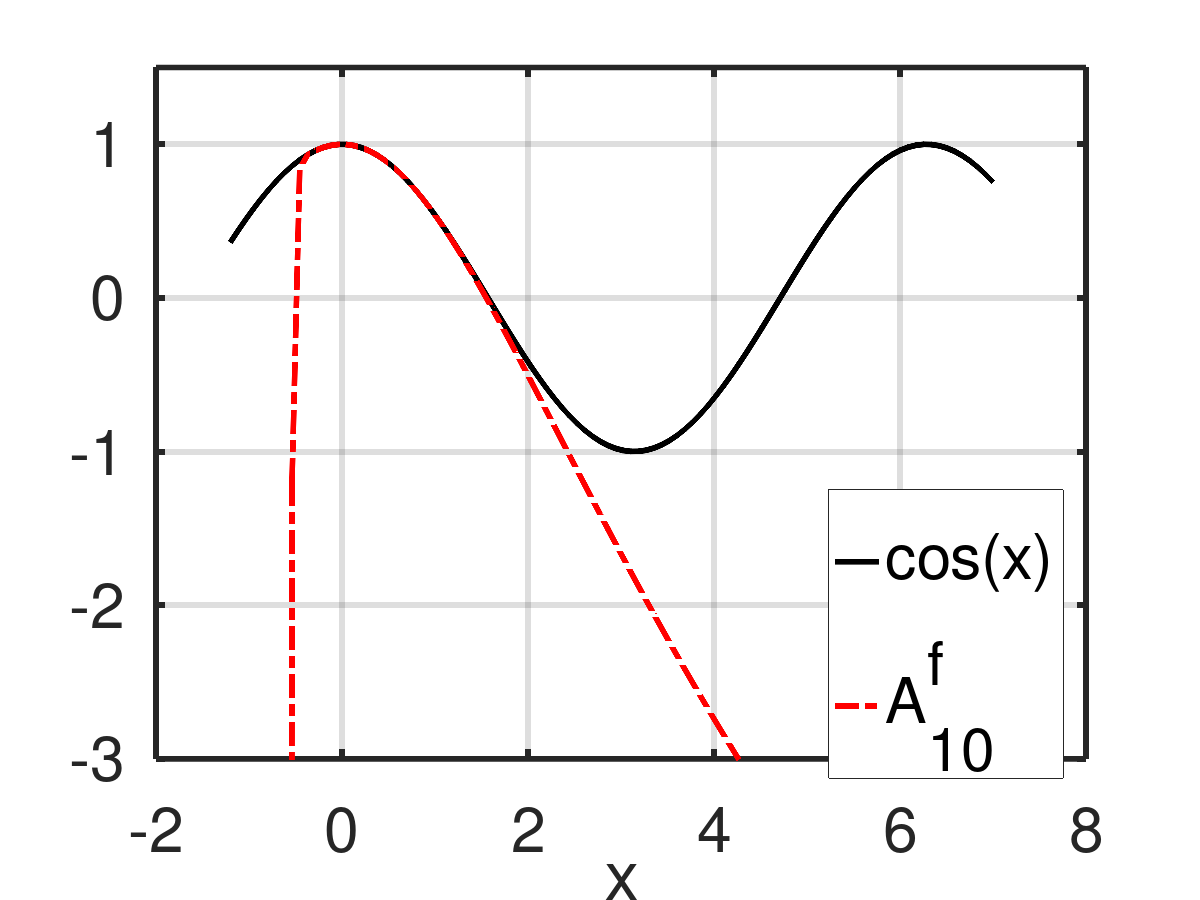}
\par\end{centering}
\begin{centering}
\includegraphics[width=0.47\linewidth]{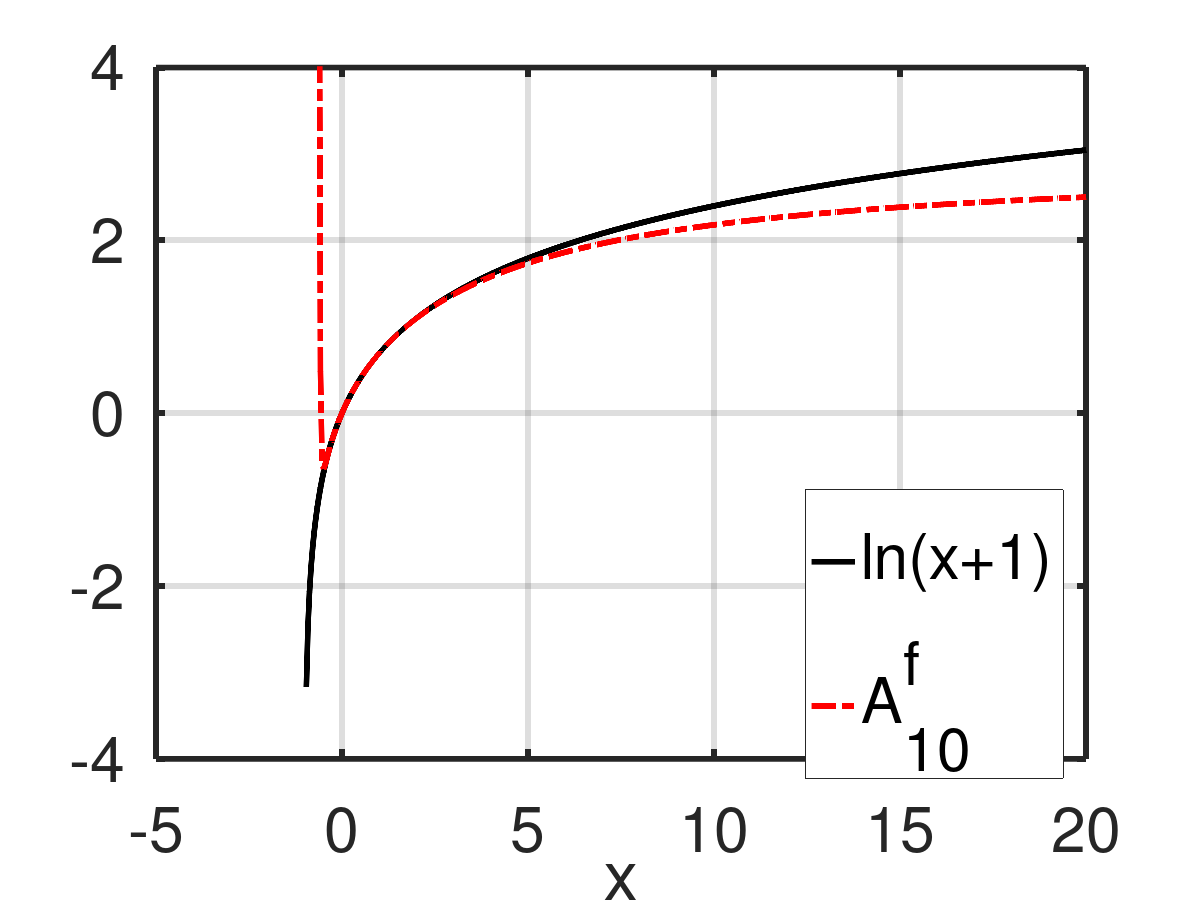} \includegraphics[width=0.47\linewidth]{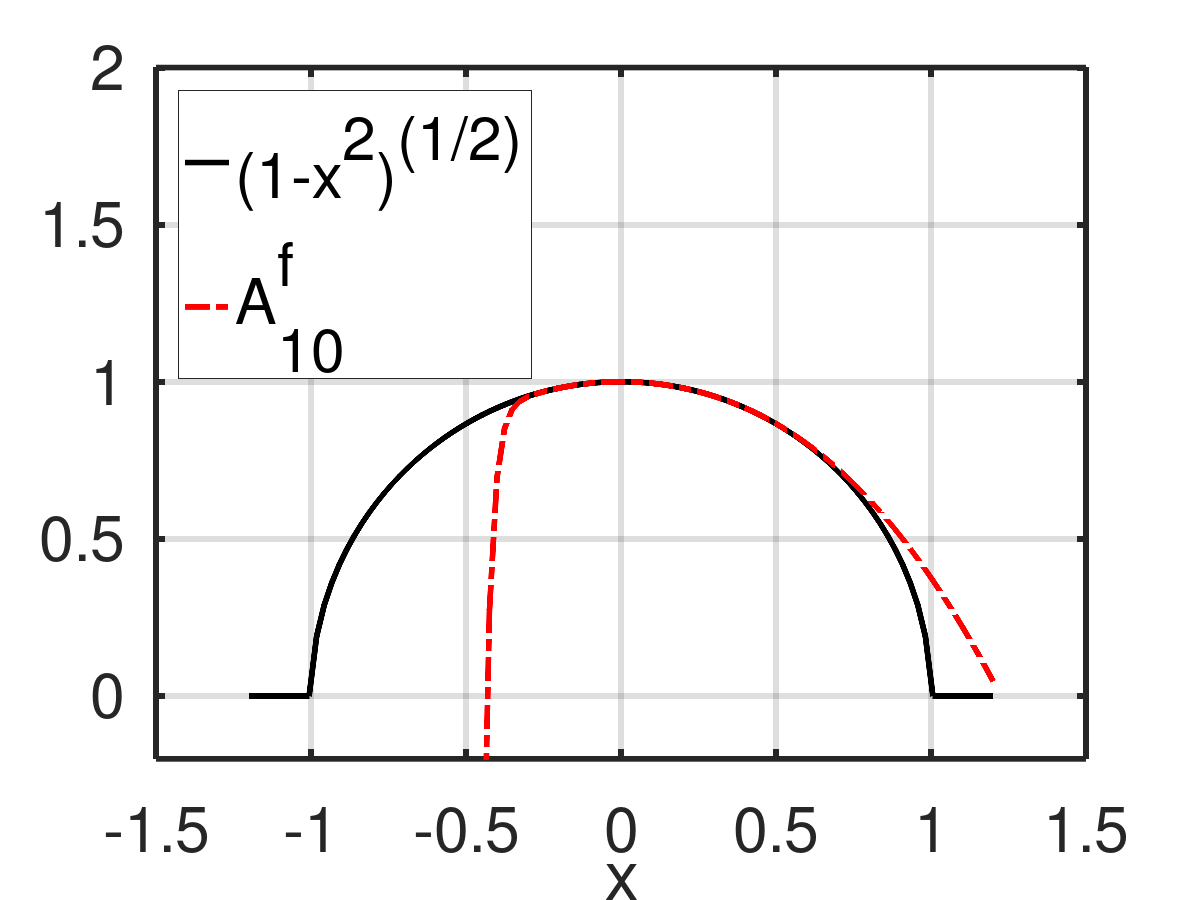}
\par\end{centering}
\caption{Illustrative examples of approximations based on powers of $g\left(x\right)=x/(x+1)$
with 11 terms.\label{fig:newPade}}
\end{figure}

\subsubsection{Expansion into $\sum a_{n}g\left(x^{n}\right),\;g'\left(0\right)\protect\neq0$\label{subsec:g(x^n)}}

An expansion of the form 
\[
\mathcal{A}^{f}\left(x\right)=\sum_{n=0}^{\infty}a_{n}g\left(x^{n}\right)\;\text{with}\;g'\left(0\right)\neq0
\]
is a triangular approximation, which follows from the chain and product
differentiation rules. The exception to the ``triangular'' behavior
is the function value for $g\left(0\right)\neq0$ cases, which requires
a dedicated treatment. Writing
\[
\mathcal{A}^{f}\left(x\right)=b_{0}+\sum_{n=1}^{\infty}a_{n}g\left(x^{n}\right)\equiv b_{0}+\mathcal{A}_{n>0}^{f}\left(x\right),
\]
one can set $b_{0}$ by hand 
\begin{equation}
b_{0}=f(0)-\mathcal{A}_{n>0}^{f}\left(0\right),\label{eq:b0_Def}
\end{equation}
where the derivative (but not value) matching approximation $\mathcal{A}_{n>0}^{f}$
can be easily constructed (as explained hereunder), since for derivatives
($n>0$) the triangular property holds. Having dealt with the $g\left(0\right)\neq0$
issue, we assume from now on a more elegant version of the expansion
with $g\left(0\right)=0$. Then $b_{0}=f\left(0\right)$ and one has
\begin{equation}
\mathcal{A}^{f}\left(x\right)=f\left(0\right)+\sum_{n=1}^{\infty}a_{n}g\left(x^{n}\right)=f\left(0\right)+\mathcal{A}_{n>0}^{f}\left(x\right).\label{eq:mobiusExpansion}
\end{equation}
Independently on the $g\left(0\right)$ value, we are now interested
in the derivative matching procedure for $\mathcal{A}_{n>0}^{f}\left(x\right)$.
Assuming the analyticity of $g$ in the neighborhood of zero 
\[
g\left(x\right)=\sum_{n=0}^{\infty}g_{n}x^{n}
\]
 we proceed with explicit calculations by regrouping the $\mathcal{A}_{n>0}^{f}$
terms by powers of $x$, as expressed in the following table
\begin{center}
\begin{tabular}{c|ccccccc}
\multicolumn{1}{c}{} &  &  &  &  &  &  & \tabularnewline
\hline 
\hline 
\multicolumn{1}{c}{} & $x$ & $x^{2}$ & $x^{3}$ & $x^{4}$ & $x^{5}$ & $x^{6}$ & \tabularnewline
\hline 
$a_{1}g\left(x\right)$ & $a_{1}g_{1}$ & $a_{1}g_{2}$ & $a_{1}g_{3}$ & $a_{1}g_{4}$ & $a_{1}g_{5}$ & $a_{1}g_{6}$ & $\cdots$\tabularnewline
$a_{2}g\left(x^{2}\right)$ & $0$ & $a_{2}g_{1}$ & $0$ & $a_{2}g_{2}$ & $0$ & $a_{2}g_{3}$ & $\cdots$\tabularnewline
$a_{3}g\left(x^{3}\right)$ & $0$ & $0$ & $a_{3}g_{1}$ & $0$ & $0$ & $a_{3}g_{2}$ & $\cdots$\tabularnewline
$a_{4}g\left(x^{4}\right)$ & $0$ & $0$ & $0$ & $a_{4}g_{1}$ & $0$ & $0$ & $\cdots$\tabularnewline
$\vdots$ & $\vdots$ & $\vdots$ & $\vdots$ & $\vdots$ & $\vdots$ & $\vdots$ & $\ddots$\tabularnewline
\hline 
\hline 
\multicolumn{1}{c}{} &  &  &  &  &  &  & \tabularnewline
\end{tabular}.
\par\end{center}

\noindent \begin{flushleft}
The factors of $x^{n}$ terms are arranged in columns, rows contain
contributions from individual summands of $\mathcal{A}_{n>0}^{f}\left(x\right)$.
By plugging different powers of $x$ (shown in different rows) into
the argument of $g$ one observes different spacing between non-zero
terms. There are no zeros in the first row, in the second row the
terms are separated by one zero, in the next row the spacing is two,
and so on. It is a version of the sieve of Eratosthenes where the
$x^{n}$ monomial is multiplied by all such $a_{i}g_{j}$ where $i$
or $j$ divides $n$ with $i\times j=n$. This is written
\[
\mathcal{A}_{n>0}^{f}\left(x\right)=\sum_{n=1}^{\infty}(\sum_{k=1}^{n}\delta_{n\%k,0}a_{k}g_{\frac{n}{k}})x^{n}\equiv\sum_{n=1}^{\infty}(\sum_{k|n}^{n}a_{k}g_{\frac{n}{k}})x^{n}.
\]
Assuming also the analyticity of $f$
\[
f\left(x\right)=\sum_{n=0}^{\infty}f_{n}x^{n}
\]
 and the convergence of the approximation $\mathcal{A}^{f}=f$ one
can compare the coefficients and write
\[
f_{n}=\sum_{k|n}^{n}a_{k}g_{\frac{n}{k}}.
\]
Considering $f_{n}$ and $g_{n}$ as fixed, we are interested in finding
the inverse relation which allows to expresses $a_{n}$ as their function.
The solution is known to be the inverse with respect to the Dirichlet
convolution, the latter being defined on two arithmetic function (or
sequences) as
\[
\left(u*v\right)_{n}=\sum_{k|n}^{n}u_{k}v_{\frac{n}{k}}.
\]
The definition of the inverse (denoted by $u^{-1}$) stands
\[
\left(u*u^{-1}\right)_{n>0}=\delta_{1,n},
\]
where, the existence of $u^{-1}$ requires the assumption $u_{1}\neq0\neq u_{1}^{-1}$.
With $u_{n}\equiv g_{n}$ it implies $g'\left(0\right)\neq0$. Thus
we know how to construct the expansion coefficients in (\ref{eq:mobiusExpansion})
\[
a_{n}=\sum_{k|n}\left(g^{-1}\right)_{k}f_{\frac{n}{k}},
\]
i.e. the expansion coefficients of $\mathcal{A}^{f}$ involve the
Dirichlet inverse of the power-expansion coefficients of $g$ (and
vice versa). Each pair of such sequences $\left\{ u_{n},\:\left(u^{-1}\right)_{n}\right\} $
can be used to construct two different approximation forms. One of
them uses $u_{n}$ to define $g$ (i.e. $g_{n}\equiv u_{n}$) and
$\left(u^{-1}\right)_{n}$ to define $a_{n}$, the other does the
opposite. Looking in the literature for known examples of Dirichlet
inverses we propose two of them.
\par\end{flushleft}
\begin{enumerate}
\item Sequences $u_{n}=1$ and $\left(u^{-1}\right)_{n}=\mu_{n}$ where
$\mu_{n}$ denotes the Möbius function. One then gets the two following
expansions
\begin{equation}
\mathcal{A}^{f}\left(x\right)=f\left(0\right)+\sum_{n=1}^{\infty}a_{n}G\left(x^{n}\right),\:a_{n}=\sum_{k|n}f_{\frac{n}{k}},\label{eq:powInArg_G}
\end{equation}
where $G$ is the ordinary generating function of the Möbius function\footnote{As seen form the upper bound on (\ref{eq:defOrdGenFnMob}) provided
by the geometric series $\sum_{n=1}^{\infty}\left|\mu\left(n\right)x^{n}\right|\leq\sum_{n=1}^{\infty}\left|x^{n}\right|$,
the definition of $G$ is absolutely convergent on $\left(-1,1\right)$.} 
\begin{equation}
G\left(x\right)\equiv\sum_{n=1}^{\infty}\mu_{n}x^{n}\label{eq:defOrdGenFnMob}
\end{equation}
and
\begin{equation}
\mathcal{A}^{f}\left(x\right)=b_{0}+\sum_{n=1}^{\infty}a_{n}\frac{1}{1-x^{n}},\:a_{n}=\sum_{k|n}\mu_{k}f_{\frac{n}{k}},\label{eq:powInArg_Rat1}
\end{equation}
with $b_{0}$ given by (\ref{eq:b0_Def}).
\item Sequences 
\[
u_{n}=\sin\left(n\pi/2\right)
\]
and
\[
\left(u^{-1}\right)_{n}\equiv\nu_{n}=\begin{cases}
\left(-1\right)^{\sum_{p|n}\frac{p+1}{2}} & \text{ for square-free odd }n\\
0 & \text{ else}
\end{cases},
\]
where $p$ denotes prime numbers. The proof \cite{nemes_2021} uses
the fact that $u$ represents the Dirichlet-series coefficients of
the Dirichlet beta function, which can be expressed as an Euler product
and thus easily inverted (with respect to the multiplication). The
$u^{-1}$ sequence is then given by the Dirichlet series coefficients
of the latter. We chose to present only one of the two expansions
\begin{equation}
\mathcal{A}^{f}\left(x\right)=f\left(0\right)+\sum_{n=1}^{\infty}a_{n}\frac{x^{n}}{x^{2n}+1},\:a_{n}=\sum_{k|n}\nu_{k}f_{\frac{n}{k}},\label{eq:powInArg_Rat2}
\end{equation}
the remaining variant seems not to be very interesting.
\end{enumerate}
In all cases the obvious relation to the characteristic numbers is
$f_{m}=c_{m}^{f}/m!$. The shape of the function $G\left(x\right)$
and approximations of the exponential and sine functions with formulas
(\ref{eq:powInArg_G}), (\ref{eq:powInArg_Rat1}) and (\ref{eq:powInArg_Rat2})
are shown in Fig. \ref{Fig:inArgPow}.
\begin{figure}[t]
\begin{centering}
\includegraphics[width=0.45\textwidth]{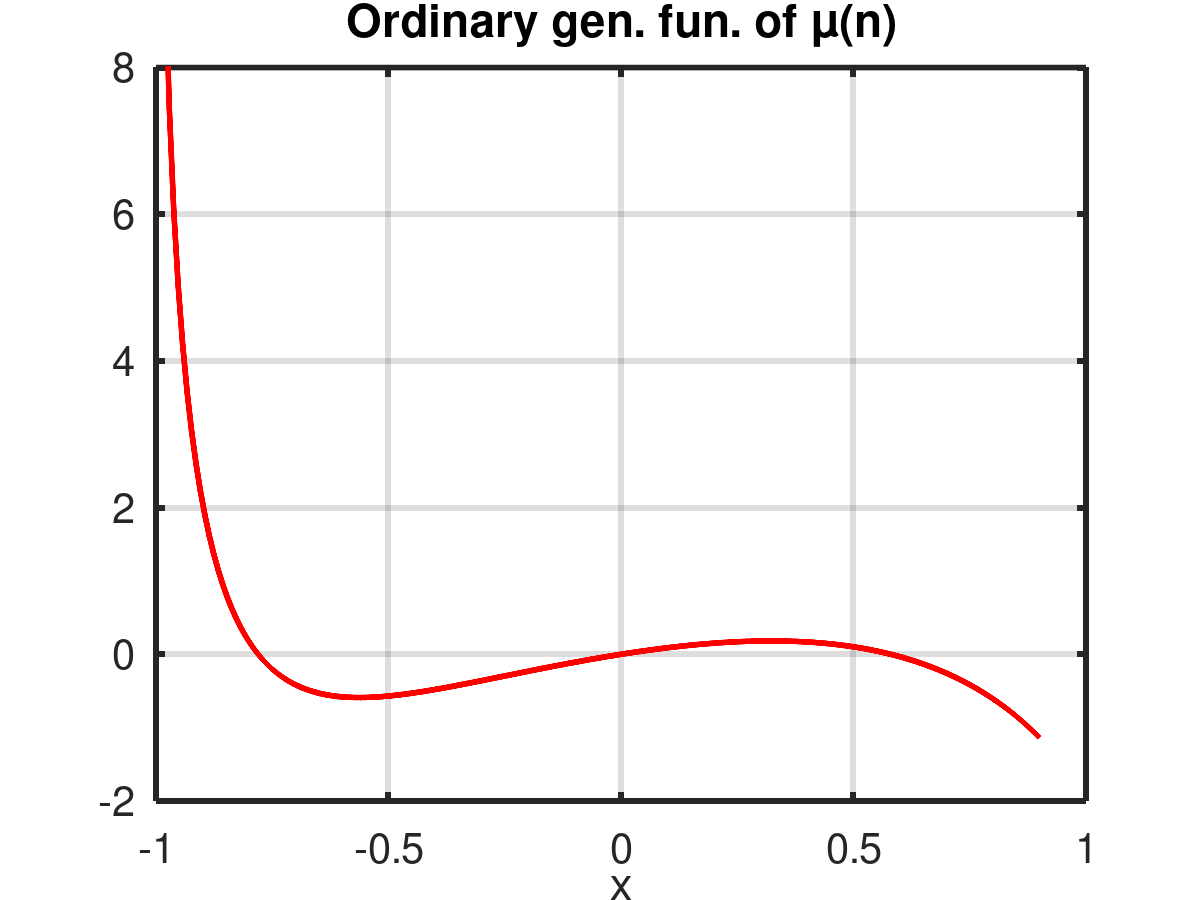}\hspace{0.5cm}\includegraphics[width=0.45\textwidth]{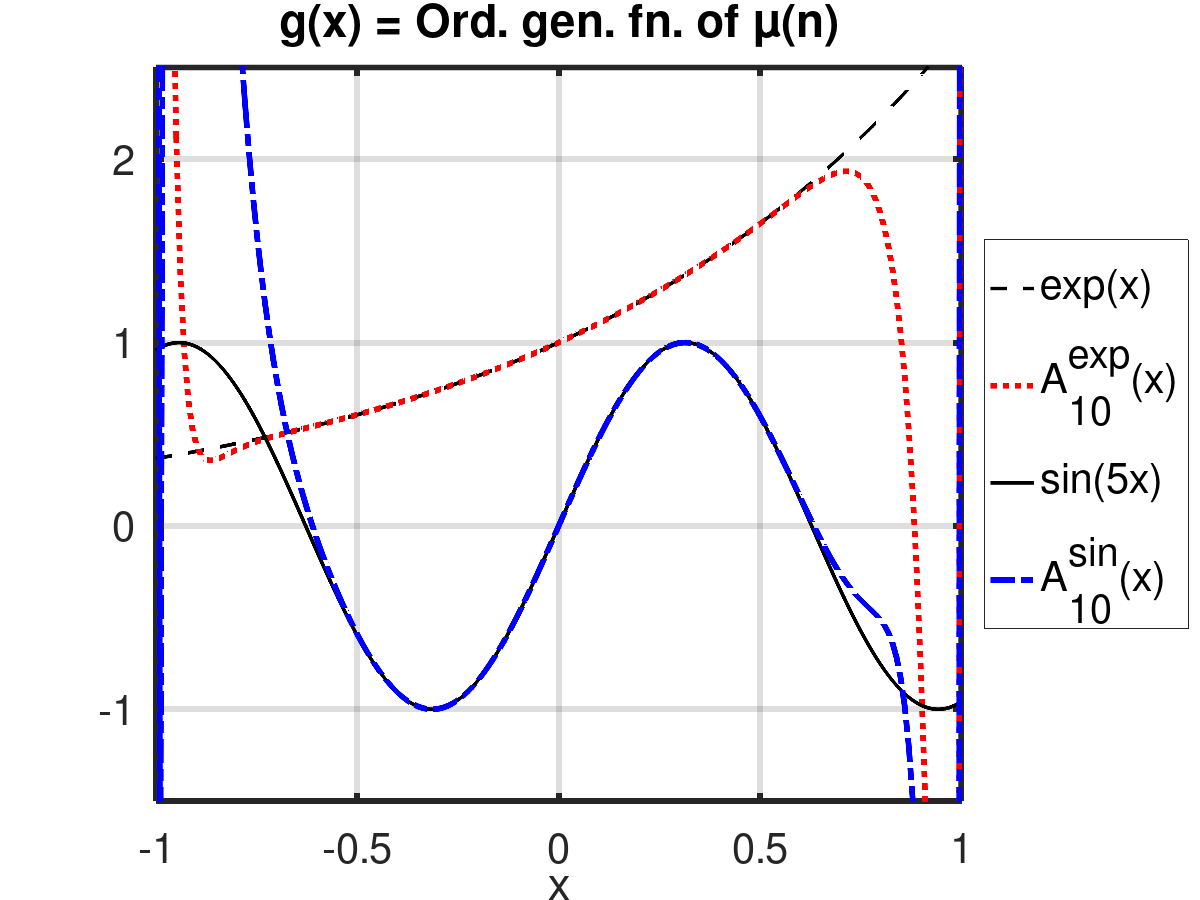}
\par\end{centering}
\begin{centering}
(a)\hspace{5cm}(b)
\par\end{centering}
\vspace{0.5cm}

\begin{centering}
\includegraphics[width=0.45\textwidth]{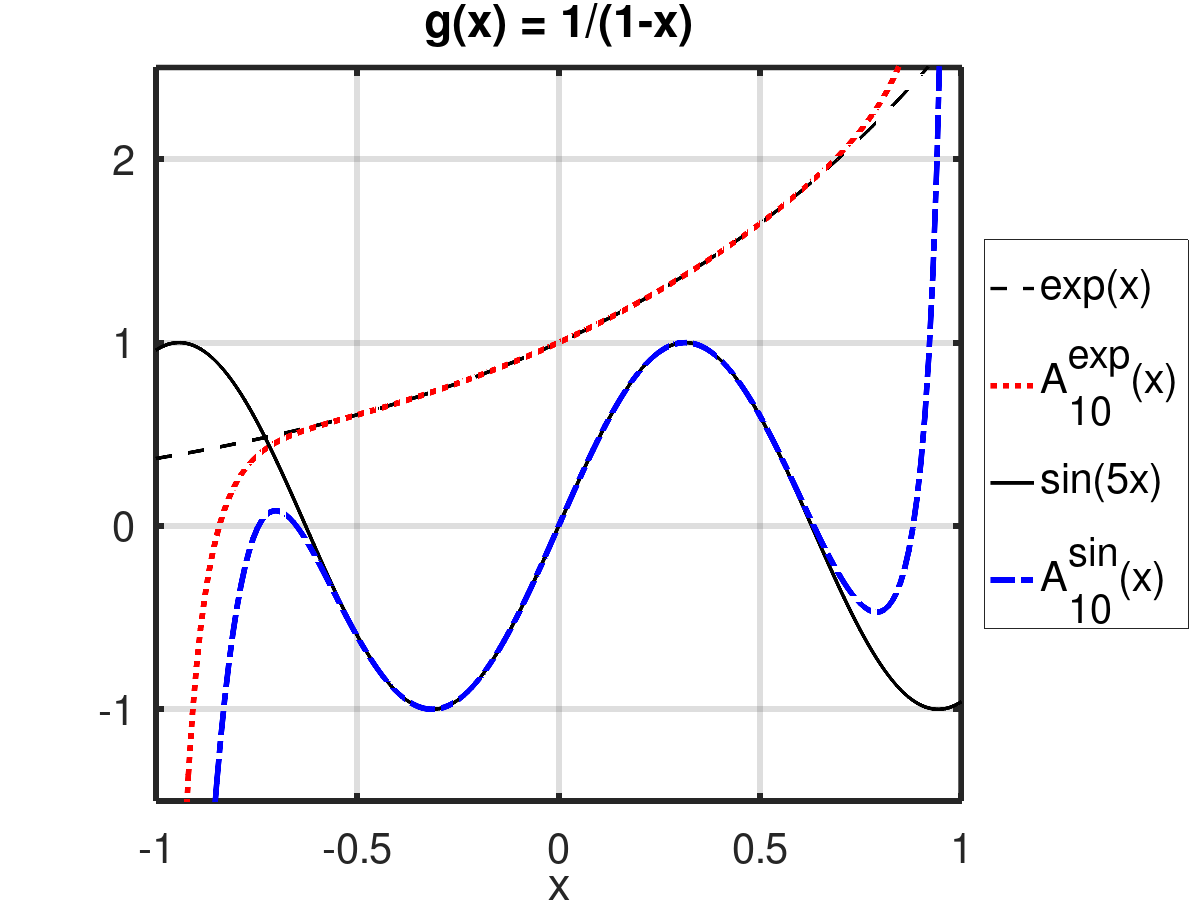}\hspace{0.5cm}\includegraphics[width=0.45\textwidth]{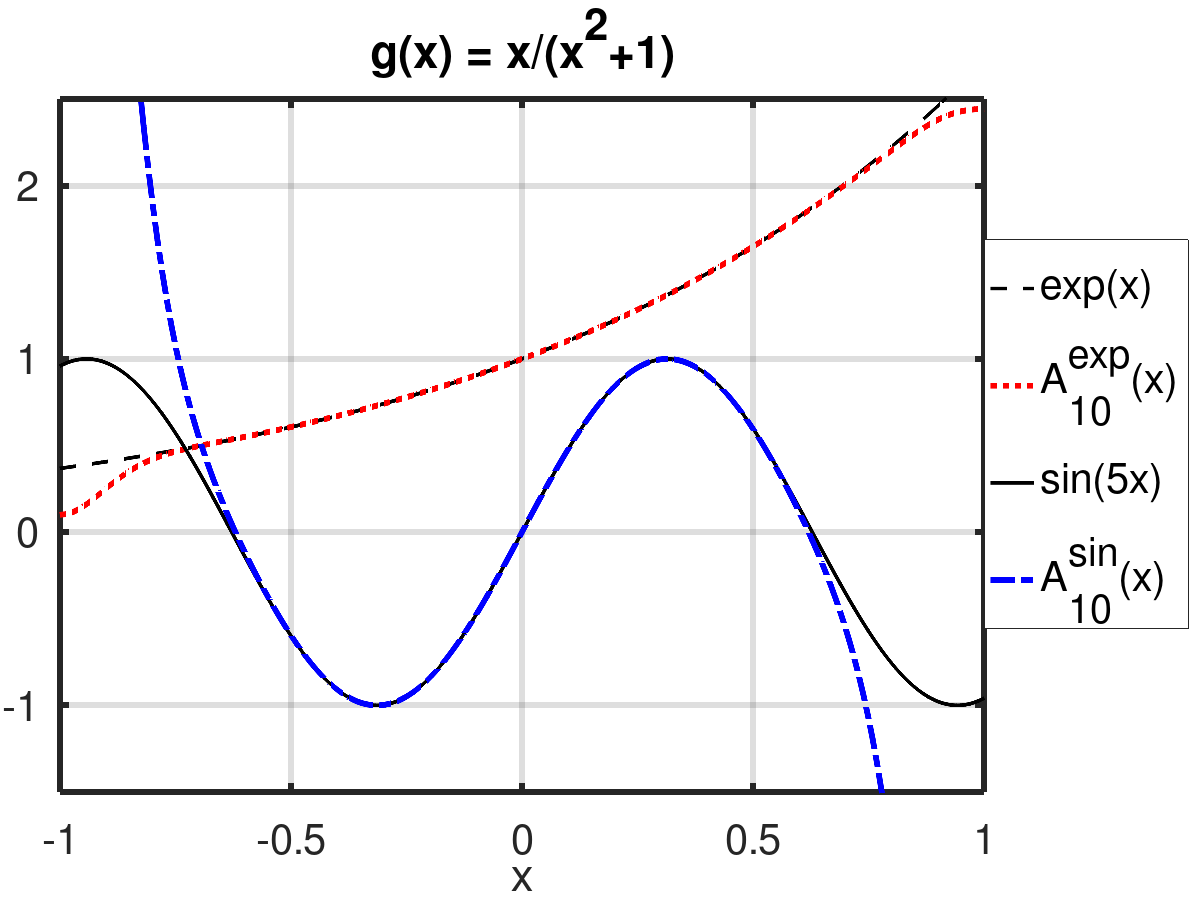}
\par\end{centering}
\begin{centering}
(c)\hspace{5cm}(d)\caption{Ordinary generating function of the Möbius function $G\left(x\right)$
on sub-figure (a) and approximations of $\exp\left(x\right)$ and
$\sin\left(5x\right)$ by $\sum_{n}^{10}a_{n}g\left(x^{n}\right)$
with $g(x)=G\left(x\right)$, $g(x)=1/\left(1-x\right)$ and $g(x)=x/\left(x^{2}+1\right)$
on sub-figures b), c) and d) respectively.}
\label{Fig:inArgPow}
\par\end{centering}
\end{figure}

The expressions (\ref{eq:powInArg_Rat1}) and (\ref{eq:powInArg_Rat2})
represent presumably new rational expansions which have singularities
situated on the unit circle in the complex plane. The positions of
the singularities of (\ref{eq:powInArg_Rat1}) correspond to all roots
of unity and thus the expression in general diverges\footnote{Unless all $a_{n}$ with even $n$ are zero, in which case it may
converge for $x=-1$.} for $x=\pm1$. The singularities of (\ref{eq:powInArg_Rat2}) are
given by even roots of $-1$, implying the individual summands being
well defined for $x=\pm1$. One can modify the radius of the circle
on which the singularities are situated by substitution, the procedure
is identical to the one described in Sec. \ref{subsec:Power-series-of}.
The domain of convergence $\mathcal{A}^{f}\rightarrow f$ remains
an open question, for various elementary functions numerical computations
suggest the interval $\left(-1,1\right)$.

\subsubsection{Decomposed exponential}

Another derivative-matching approximation can be constructed by considering
functions which form a closed ring when differentiated
\[
\left\{ h_{n}\right\} _{n=0}^{N-1}:\quad h_{n}^{'}=h_{n+1},\;h_{N-1}^{'}=h_{0},
\]
where $N$ is fixed. If $h_{i}$ should play the role of approximation
building blocks, their mutual linear independence is suitable. This
is however not automatically satisfied (consider the function $\sin(x)$
and the corresponding derivative ring), yet, from the theory of linear
differential equations we know the equation 
\[
h^{\left(N\right)}\left(x\right)=h\left(x\right)
\]
 allows for $N$ independent solutions. In search of them one can
use the fact that the derivative ring as whole (i.e. summed) is derivative-invariant
and thus has to be proportional to the exponential function
\[
\sum_{n=0}^{N-1}h_{n}\left(x\right)=\alpha\exp\left(x\right),
\]
from where the idea to decompose and re-arrange the power expansion
of the latter
\[
\text{dex}_{\left[N,n\right]}\left(x\right)=\sum_{k=0}^{\infty}\frac{1}{\left(n+kN\right)!}x^{n+kN},\quad0\leq n<N.
\]
Convergence properties of this definition are easy to asses: on the
positive real axis all $\text{dex}$ function expansions take the
form of a sum of positive numbers majorated by the exponential and
thus necessarily convergent. Since the convergence domain is a disk
around the expansion point, all functions converge in the whole complex
plane.

One recovers some known functions
\[
\text{dex}_{\left[1,0\right]}\left(x\right)=\exp\left(x\right),\quad\text{dex}_{\left[2,0\right]}\left(x\right)=\cosh\left(x\right),\quad\text{dex}_{\left[2,1\right]}\left(x\right)=\sinh\left(x\right).
\]
The $\text{dex}$ functions have an important property: only the first
in the series $\text{dex}_{\left[N,0\right]}$ is non-zero at zero,
all others take by definition a zero value. Because of the ring property
(each member becomes its neighbor when differentiated), the differentiation
makes the only non-zero member shift along the ring. Thus, considering
the value and first $N-1$ derivatives $\left\{ c_{n}^{f}\right\} _{n=0}^{N-1}$
of a function, an approximation can be easily constructed as
\begin{equation}
\mathcal{A}_{N}^{f}\left(x\right)=\sum_{n=0}^{N-1}c_{n}^{f}\text{dex}_{\left[N,n\right]}\left(x\right).\label{eq:derRingApprox}
\end{equation}
The formula can be of course shifted to an arbitrary point by shifting
the argument. The expression (\ref{eq:derRingApprox}) is to be used
as a partial approximation, it suits situations where the number of
derivatives to match is fixed and given in advance because going to
a higher approximation order means using a different set of the $\text{dex}$
functions. The derivatives of (\ref{eq:derRingApprox}) are by construction
cyclic with period $N$. The convergence to the expanded function
can be studied in the limit $N\rightarrow\infty$ where, as a special
case, one obtains the Taylor series whose convergence properties are
well understood.

As a curiosity, one can define a function (see Fig. \ref{PprimePic})
\begin{figure}
\centering{}\includegraphics[width=0.6\textwidth]{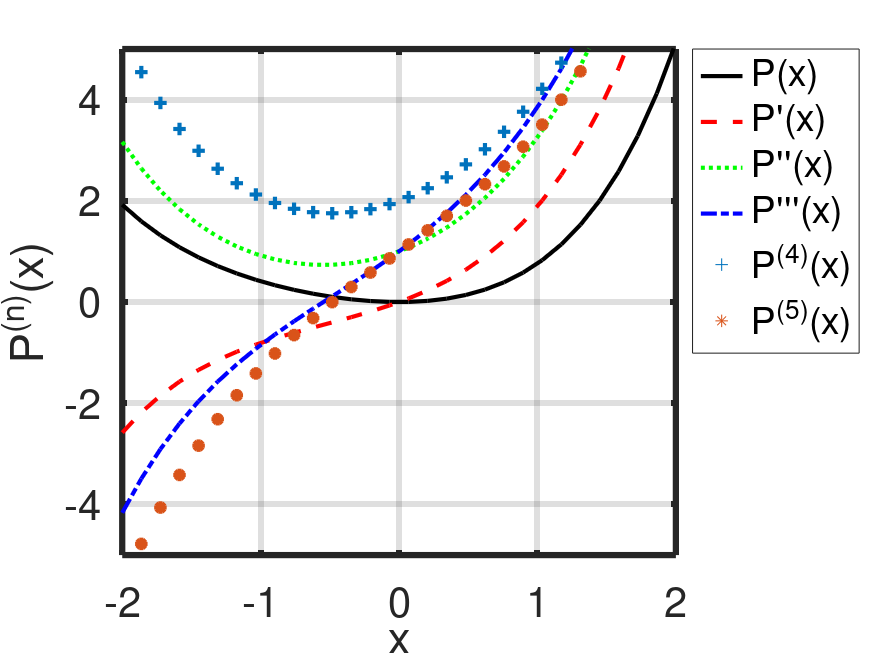}\caption{$P\left(x\right)$ function with its derivatives (see the text).}
\label{PprimePic}
\end{figure}
\[
P\left(x\right)=\sum_{i=2}^{\infty}\left(\text{dex}_{\left[i,0\right]}-1\right),
\]
where the summation is performed in the first index. It is a realization
of the sieve of Eratosthenes, one has
\[
p\text{ is prime}\Leftrightarrow\frac{d^{p}}{dx^{p}}P\left(x\right)|_{x=0}=1.
\]

\subsection{Non-linear approximation example}

\begin{figure}[t]
\begin{centering}
\includegraphics[width=0.45\linewidth]{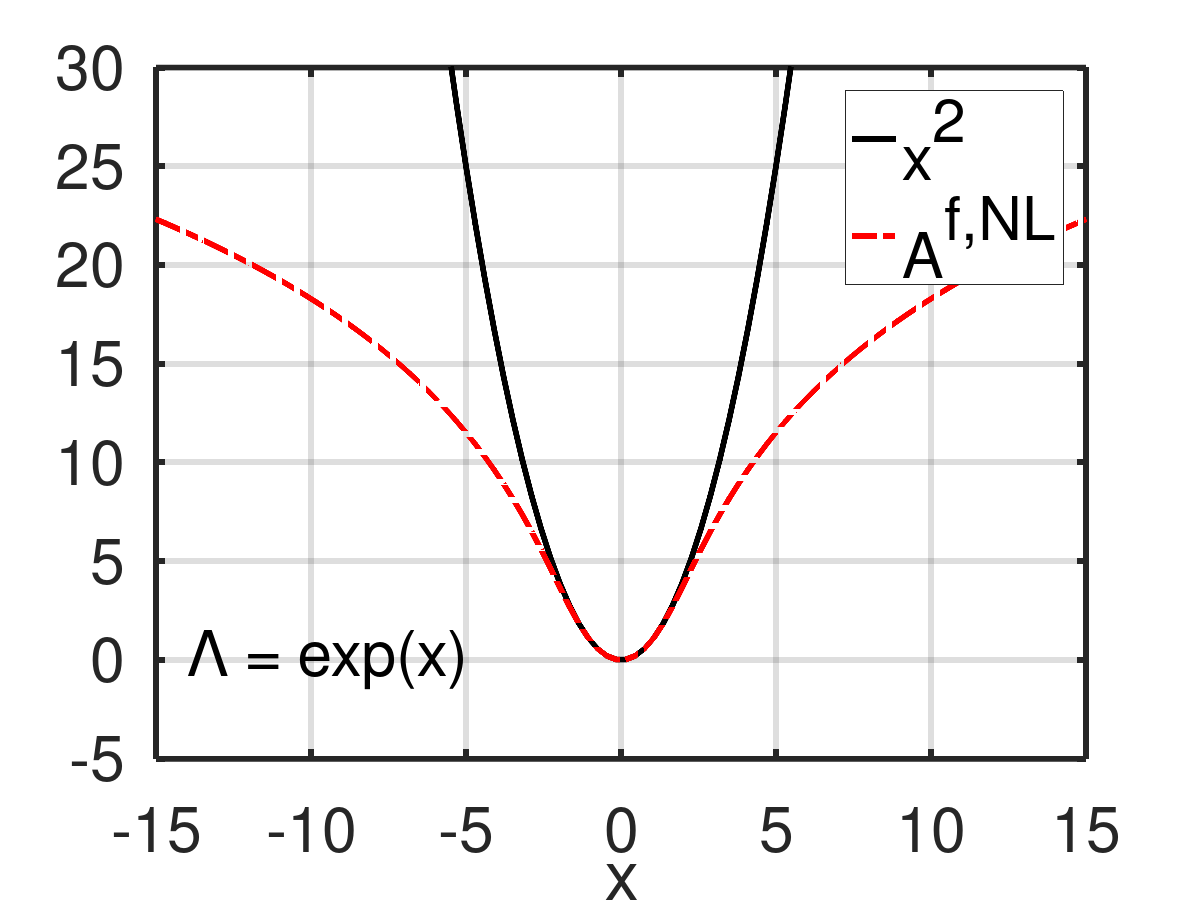}\enskip{}\includegraphics[width=0.45\linewidth]{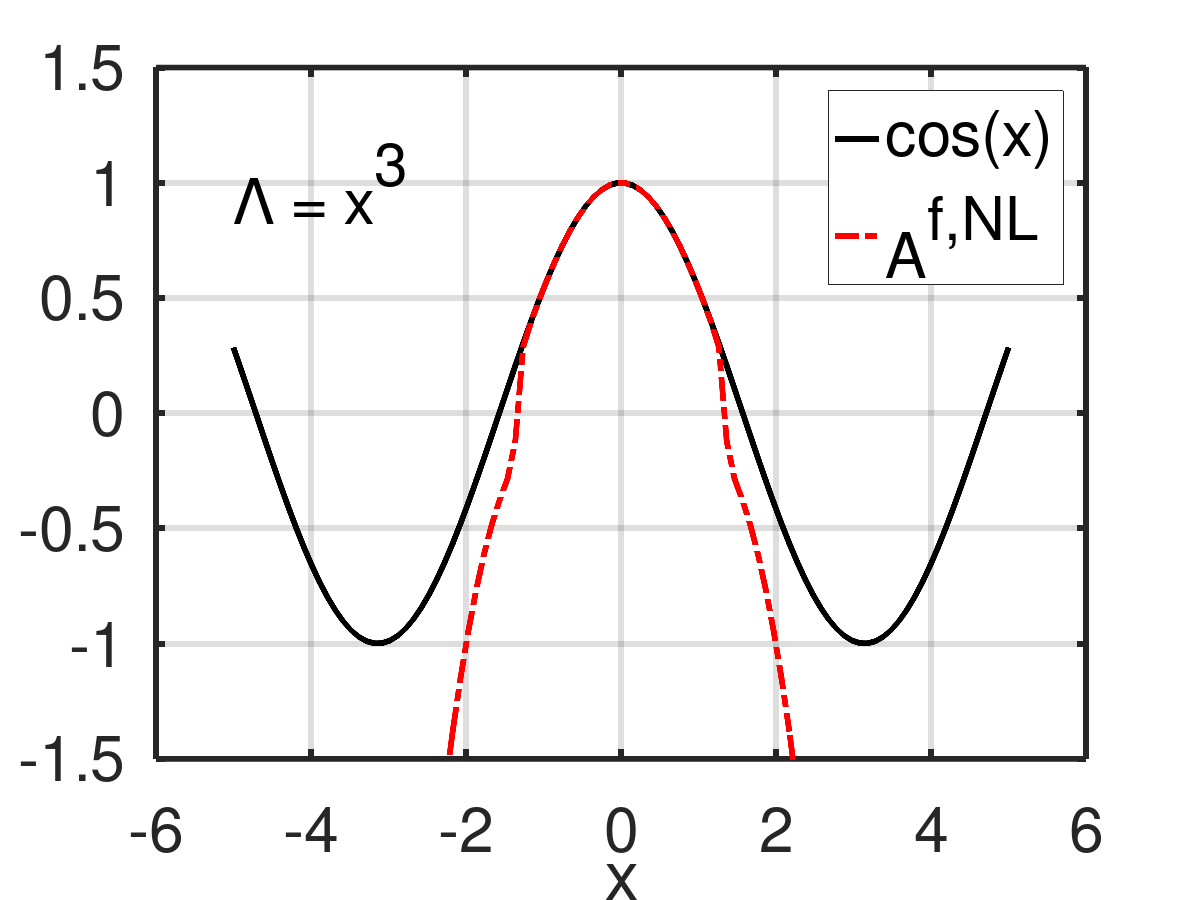}\\
\includegraphics[width=0.45\linewidth]{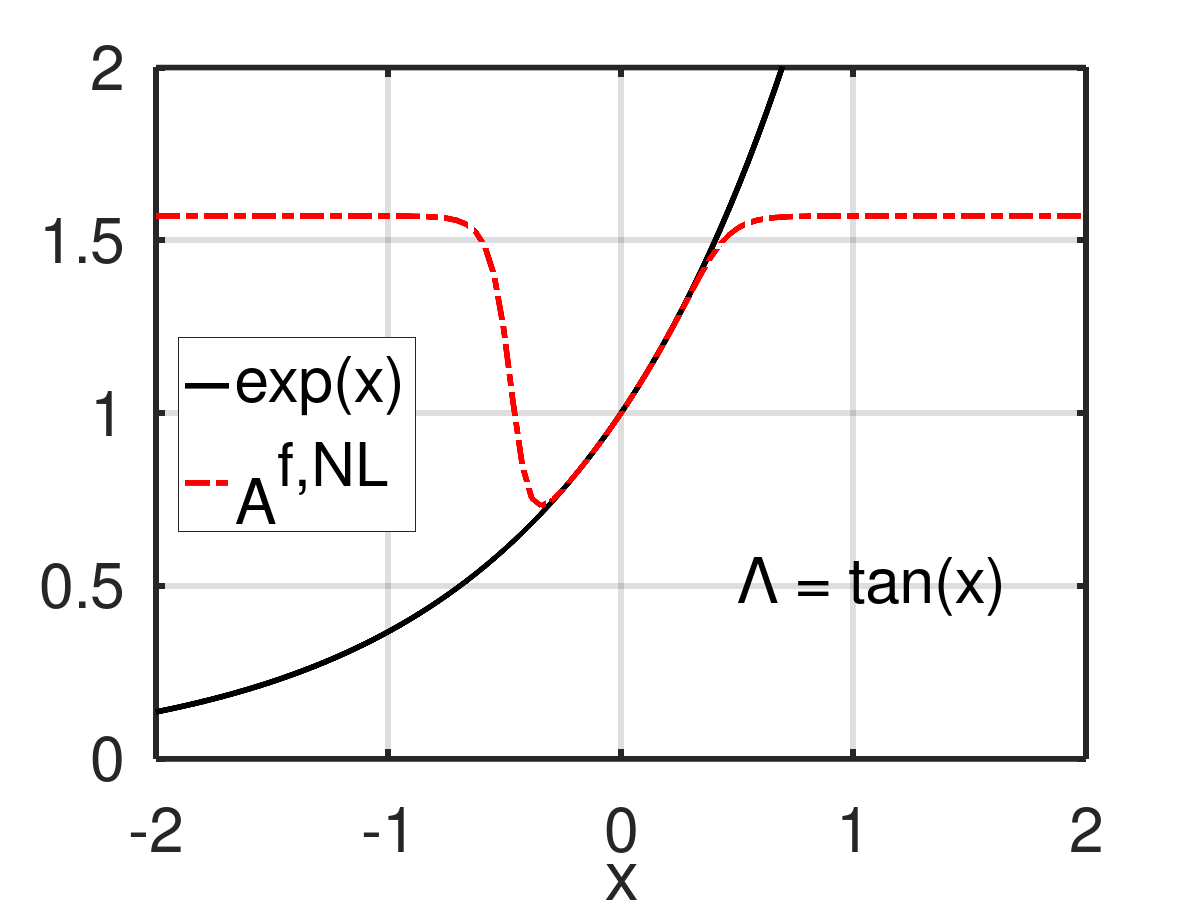}\enskip{}\includegraphics[width=0.45\linewidth]{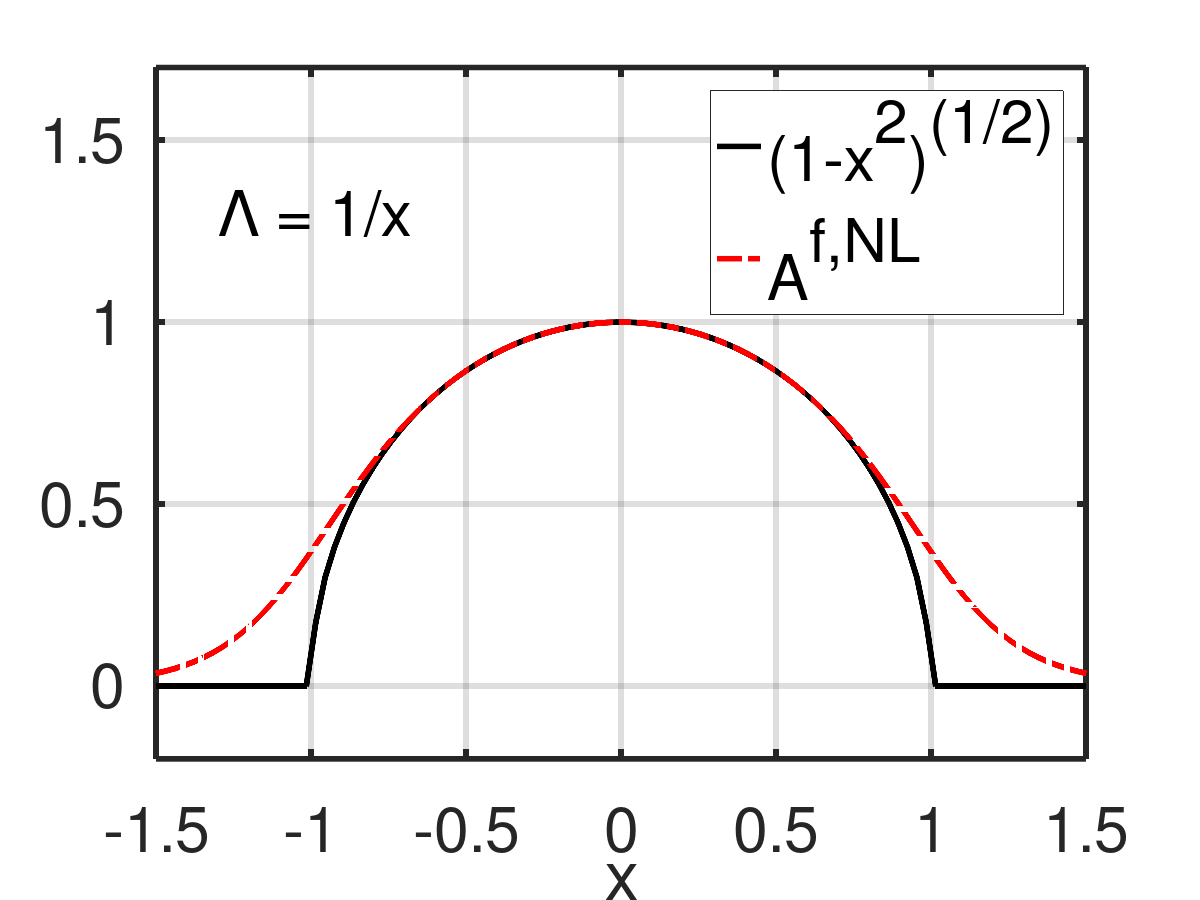}
\par\end{centering}
\caption{Illustrative examples of the nonlinear approximation (\ref{eq:nonLinApp})
with 11 terms.}
\label{pic:nonLin}
\end{figure}
A simple example of a nonlinear delta approximation can be constructed
by matching derivatives of a function $f$ transformed by a different
function $\Lambda$
\[
\mathcal{C}_{n}=\frac{d^{n}}{dx^{n}}\Lambda\left[\ldots\right]|_{x=0},\quad c_{n}^{f}=\frac{d^{n}\Lambda\left[f\left(x\right)\right]}{dx^{n}}|_{x=0}.
\]
We choose by purpose $\Lambda$ to be non-linear so as to imply the
non-linearity of the whole functional. The function $f$ is assumed
to be strictly monotonic on $\left(a,b\right)$ and $\Lambda$ invertible
$\Omega=\Lambda^{-1}$ on its codomain. The approximation is then
constructed as
\begin{equation}
\mathcal{A}^{f,NL}\left(x\right)=\Omega\left(\sum_{n=0}^{\infty}a_{n}x^{n}\right)\text{ with }a_{n}=c_{n}^{f}.\label{eq:nonLinApp}
\end{equation}
The construction is indeed almost trivial, yet it provides us with
a valid example of a non-linear functional. Clearly, the convergence
properties depend on the convergence behavior of the power expansion
which can be analyzed using standard tools known from the Taylor series.
This expansion might be of some interest: Setting by hand $a_{n}=\delta_{n,1}$
one has $\mathcal{A}^{\bcancel{f},NL}\left(x\right)=\Omega\left(x\right)$,
and so, progressively changing (in some way) the coefficients leads
to a progressive change in the function form from $\Omega\left(x\right)$
to $f\left(x\right)$ (assuming the convergence of $\mathcal{A}^{f,NL}$
to $f$)
\[
\left\{ a_{n}=\delta_{n,1}\right\} \rightarrow\left\{ a_{n}=c_{n}^{f}\right\} \Rightarrow\Omega\left(x\right)\rightarrow f\left(x\right).
\]
This can be used in a model comparison: if, in some area of science,
an established model predicting the behavior $y=\Omega\left(x\right)$
is to be improved by a more precise model, it might be natural to
express the prediction of the new model in the form $\Omega\left(\sum_{n=0}^{\infty}a_{n}x^{n}\right)$,
so that the deviations of coefficients from zero (from one for $x_{0}$)
encode the deviation of the new model from the old one.

Some illustrative examples of the approximation based on formula (\ref{eq:nonLinApp})
are shown in Fig. \ref{pic:nonLin}.

\subsection{Variations on Whittaker--Shannon formula\label{subsec:Variations-on-Whittaker=002013Shannon}}

The Whittaker--Shannon (WS) formula
\[
f\left(x\right)=\lim_{t\rightarrow x}\sum_{n=-\infty}^{+\infty}f\left(\frac{n\pi}{T}\right)\frac{\sin\left(Tt-n\pi\right)}{Tt-n\pi},\quad\widetilde{f}\left(\omega\right)=0\text{ for }\left|\omega\right|>T,
\]
where $\widetilde{f}\left(\omega\right)$ denotes the Fourier transform
assumed continuous on $\left[-T,T\right]$, can be generalized to
yield various interpolations and integral-matching approximations
(these can be interpreted as interpolations of the primitive function).
Such approximations share the argument-scaling feature, i.e. the function
is evaluated at different points (or the integral is computed in different
limits). In this section we generalize the WS formula to allow for
various point distributions, including, for example, distributions
on finite intervals concentrated around limit points. We believe this
is interesting because the generalization we propose significantly
increases the number of use cases (i.e. point configurations). We
do not study the effect of the generalization on convergence properties,
with this respect we present only numerical observations.

In full generality one can consider a ratio 
\begin{equation}
\mathcal{A}^{f,WS}\left(x\right)=\lim_{t\rightarrow x}\sum_{n}a_{n}\frac{\lambda_{n}\left(t\right)}{\omega_{n}\left(t\right)}\label{eq:WhittSh_ApxGen}
\end{equation}
where for an infinite set of distinct real numbers $\left\{ x_{n}\right\} _{n=-\infty}^{\infty}$
the following is true
\[
\lambda_{m}\left(x_{n}\right)=0,\quad\omega_{m}\left(x_{n\neq m}\right)\neq0,\quad\lim_{t\rightarrow x_{n}}\frac{\lambda_{n}\left(t\right)}{\omega_{n}\left(t\right)}=1.
\]
The lower index in $\lambda_{n}$ appears for generality reasons but
does not represent a key idea of the construction. A modified version
of the above formula 
\[
\lambda\left(t\right)=\lambda_{0}\left(t\right)=\lambda_{1}\left(t\right)=\lambda_{2}\left(t\right)=\ldots
\]
may be more appealing for its simplicity (denoted later on without
any index). It is straightforward to see that the approximation property
holds
\[
a_{n}=c_{n}^{f}\equiv f\left(x_{n}\right),
\]
because the individual terms in (\ref{eq:WhittSh_ApxGen}) represent
delta functions
\[
\lim_{t\rightarrow x_{m}}\frac{\lambda\left(t\right)}{\omega_{n}\left(t\right)}=\delta_{n,m},
\]
from which a delta approximation is built. There are no many candidates
for $\lambda$ when searching among elementary or commonly used functions:
one naturally considers the trigonometric functions, eventually also
the Bessel functions, and their modifications.

We propose the function form
\begin{equation}
\mathcal{A}_{N}^{f,WS}\left(x\right)\doteq\sum_{n=-N}^{-N}a_{n}\frac{\mathcal{N}\left(x\right)\sin\left[\pi s\left(x\right)\right]}{s_{n}\left(x-x_{n}\right)},\quad s_{n}\equiv\frac{d}{dx}\left\{ \mathcal{N}\left(x\right)\sin\left[\pi s\left(x\right)\right]\right\} |_{x=x_{n}},\label{eq:genWS}
\end{equation}
where we use $\doteq$ to denote the limit when $x\rightarrow x_{n}\equiv s^{-1}\left(n\right)$,
$s\left(x\right)$ is an argument-scaling function invertible on some
non-zero interval $I$, $\left\{ x_{n}\right\} \subset I$ and $\mathcal{N}\left(x\right)$
is a normalization factor. The latter is important mainly in the neighborhood
of limit points, e.g. the function $\sin\left(\pi/x\right)$ behaves
chaotically in the proximity of zero while $x\sin\left(\pi/x\right)$
behaves more nicely and this has an impact also on the corresponding
sums. As illustrations we provide the following examples
\begin{figure}
\begin{centering}
\includegraphics[width=0.48\textwidth]{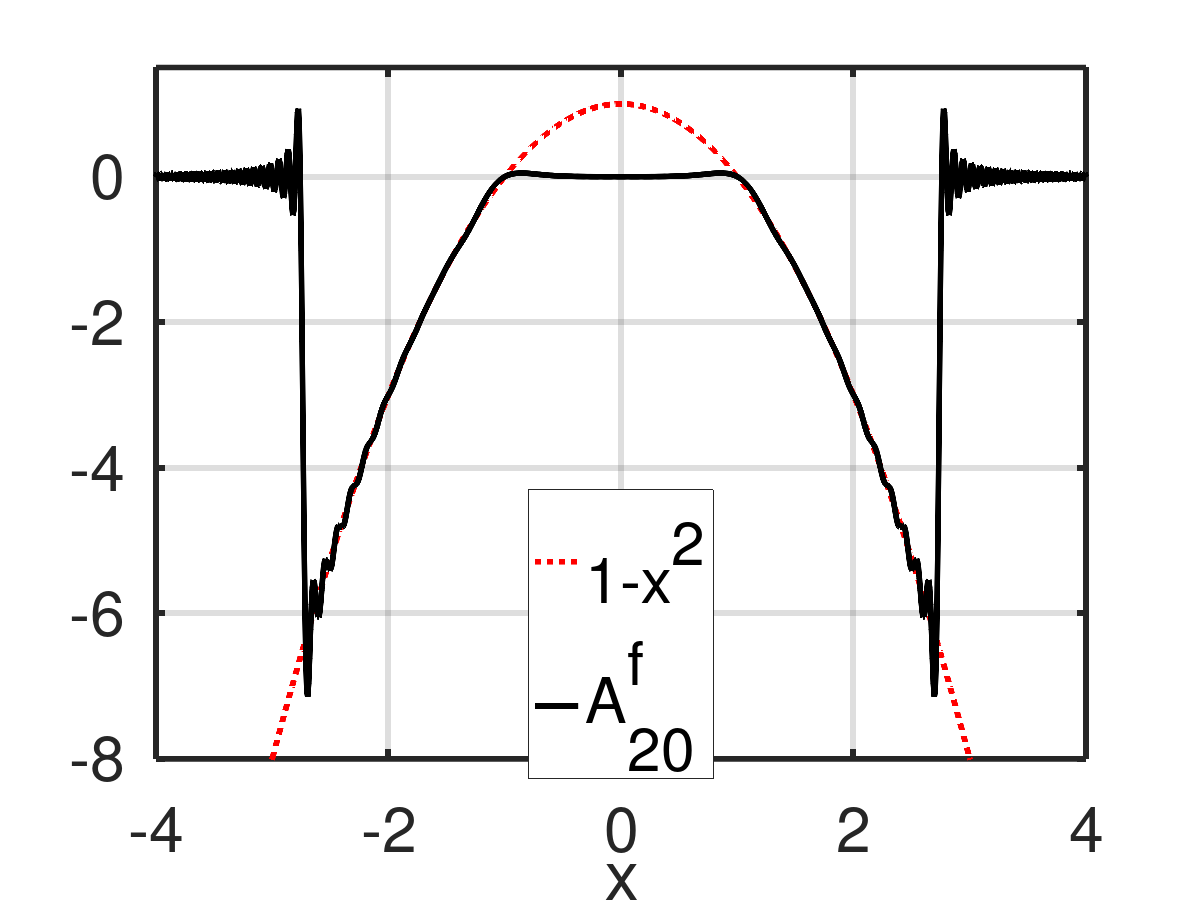}\includegraphics[width=0.48\textwidth]{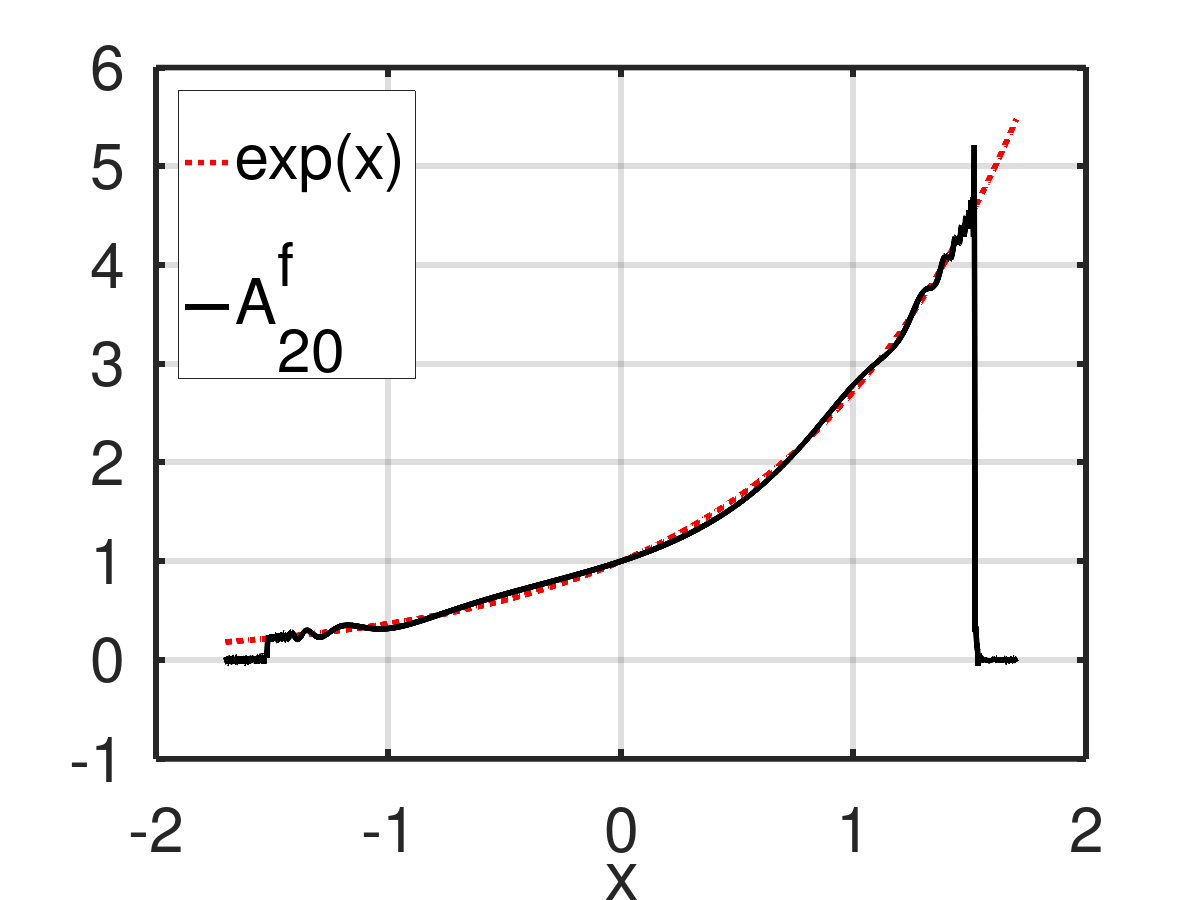}\\
(a)\hspace{0.45\textwidth}(b)\\
\bigskip{}
\par\end{centering}
\begin{centering}
\includegraphics[width=0.48\textwidth]{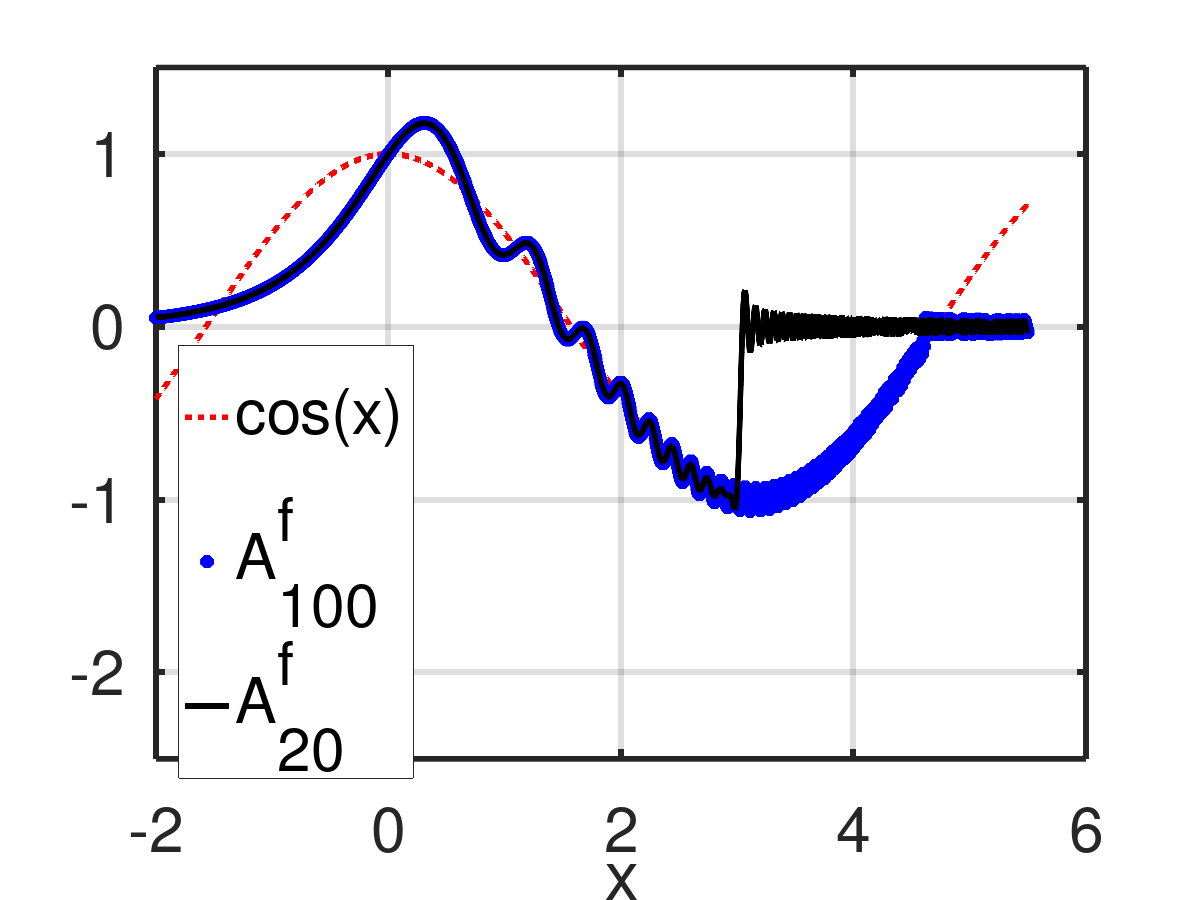}\includegraphics[width=0.48\textwidth]{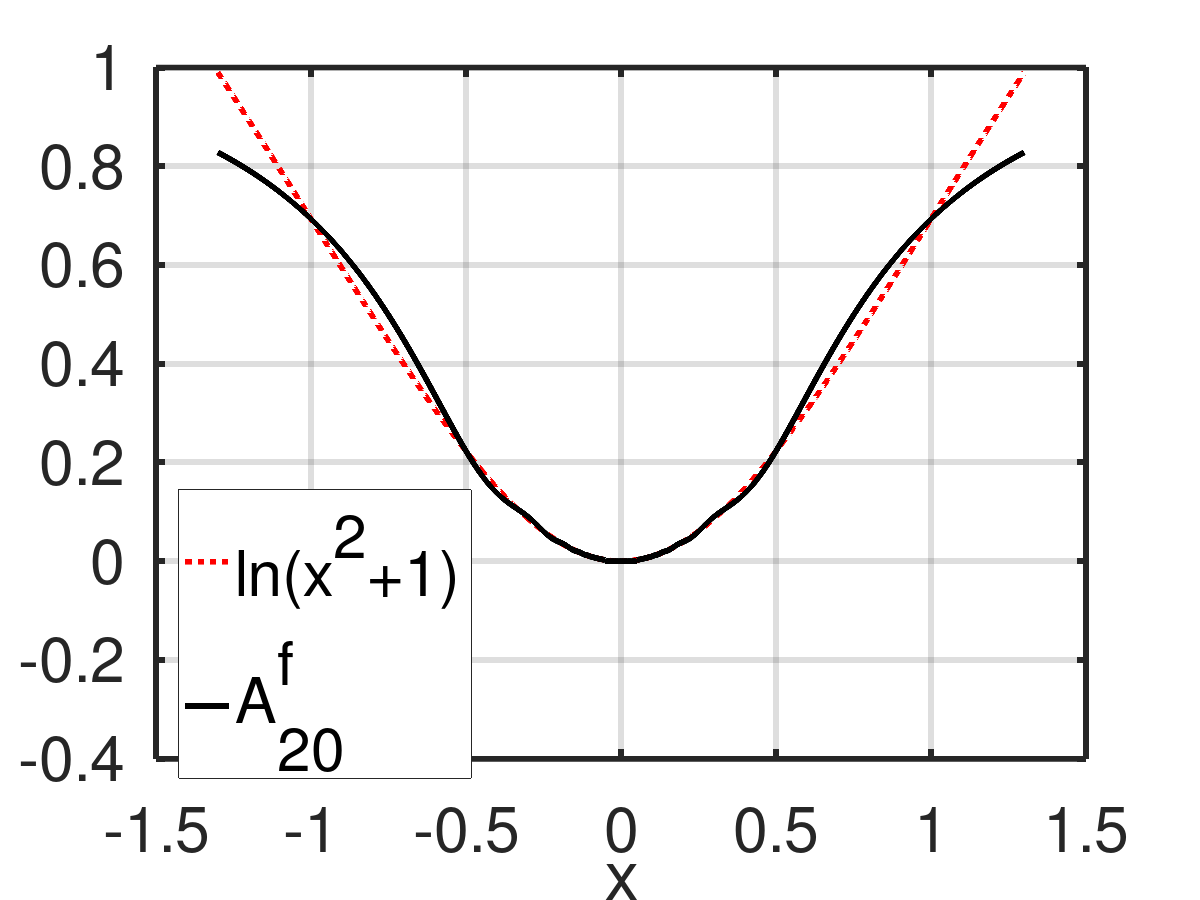}\\
(c)\hspace{0.45\textwidth}(d)\\
\bigskip{}
\par\end{centering}
\begin{centering}
\includegraphics[width=0.48\textwidth]{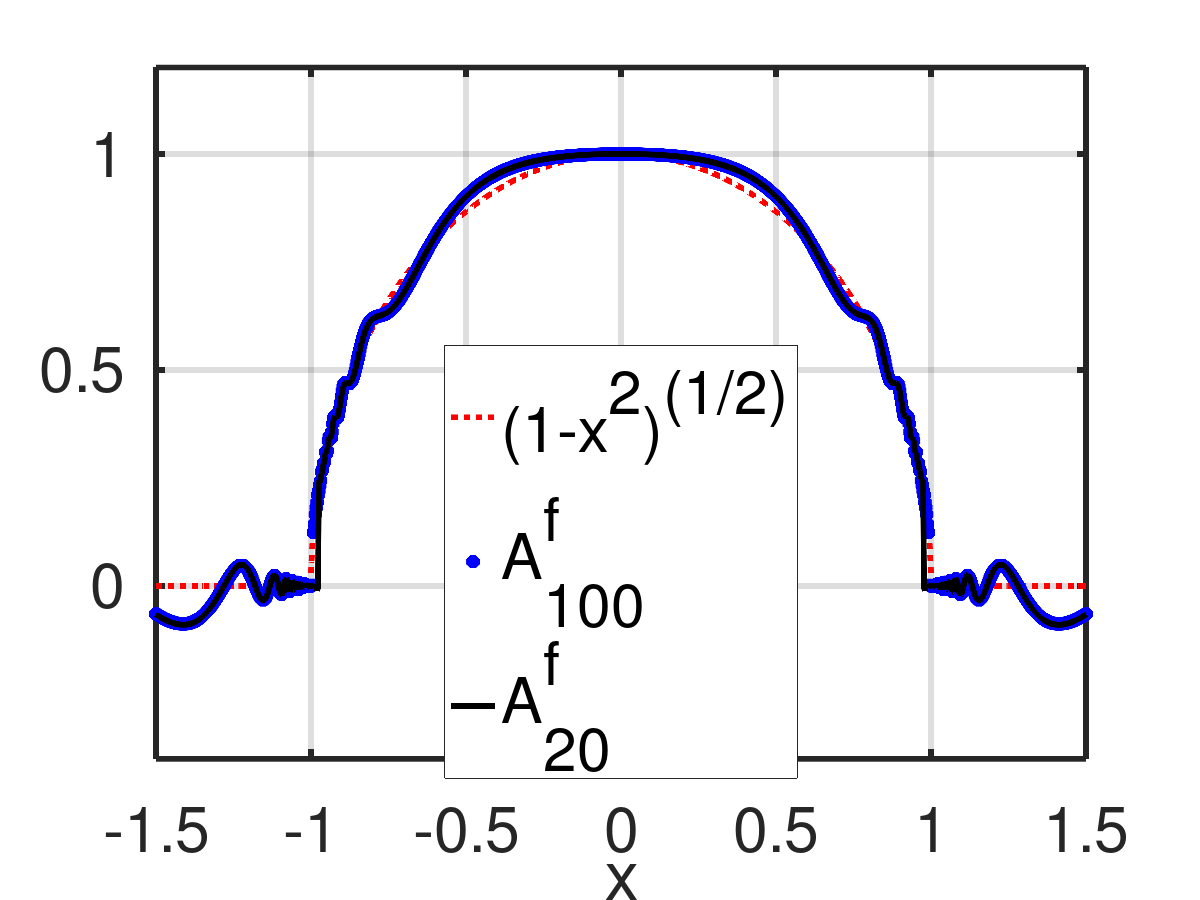}\includegraphics[width=0.48\textwidth]{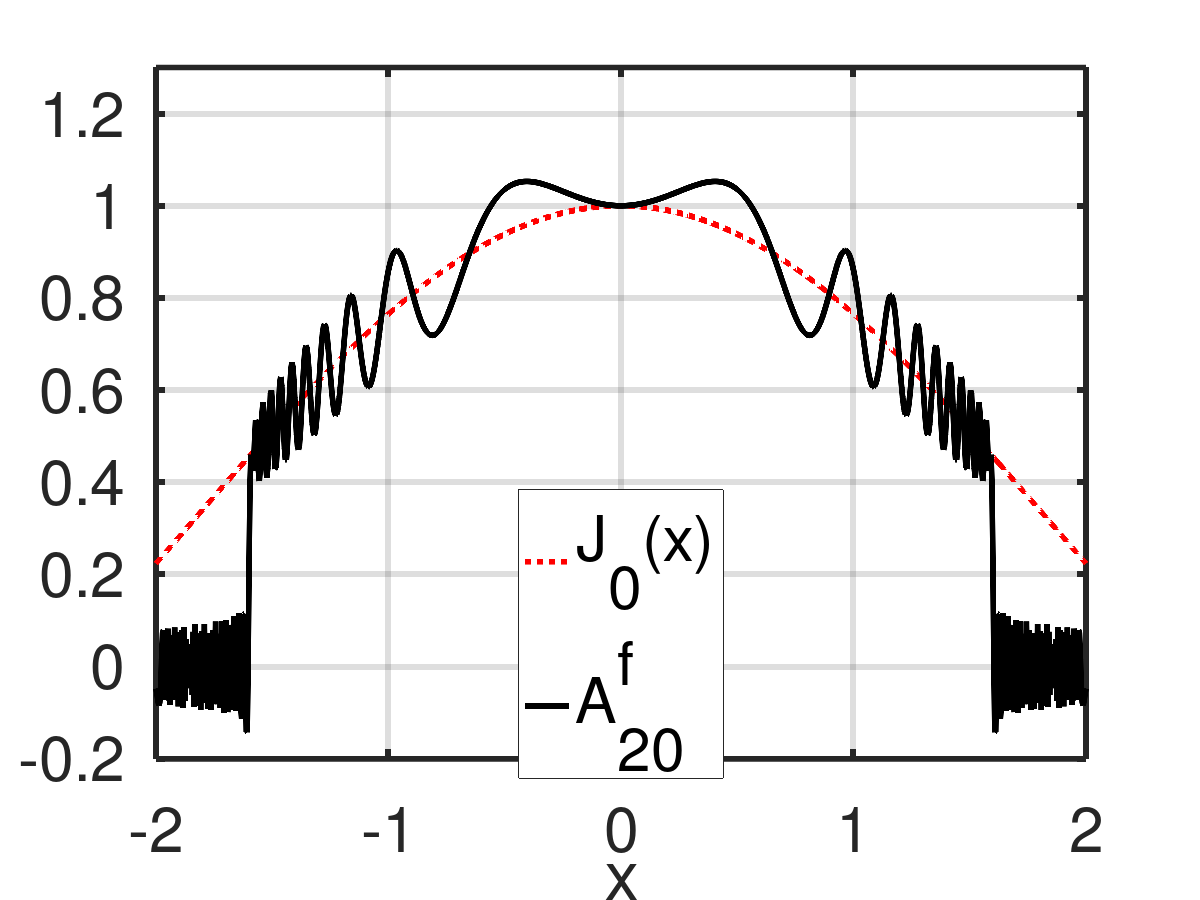}\\
(e)\hspace{0.45\textwidth}(f)
\par\end{centering}
\caption{Approximations based on a modified WS formula (for details see the
text).}
\label{Fig:wsPics}

\end{figure}

\begin{lyxlist}{00.00.0000}
\item [{(a)}] We approximate $f\left(x\right)=1-x^{2}$ using $s(x)=x^{3}$,
$x_{n}=\sqrt[3]{n}$, $s_{n}=3\pi n^{2/3}\left(-1\right)^{n}$, $\mathcal{N}\left(x\right)=1$,
$N=20$. Because the derivative of $s$ at zero is zero, the term
$n=0$ is excluded from the sum\footnote{The denominator in (\ref{eq:genWS}) cannot compensate for the numerator
and provide a finite limit at $x_{0}=0$ since it is only linear.}. The approximation is show in Fig. \ref{Fig:wsPics}(a).
\item [{(b)}] We approximate $f\left(x\right)=e^{x}$ using $s(x)=\tan\left(x\right)$,
$x_{n}=\arctan\left(n\right)$, $s_{n}=\pi\left(n^{2}+1\right)\left(-1\right)^{n}$,
$\mathcal{N}\left(x\right)=1$, $N=20$. The approximation is show
in Fig. \ref{Fig:wsPics}(b).
\item [{(c)}] We approximate $f\left(x\right)=\cos\left(x\right)$ using
$s(x)=e^{x}$, $x_{n}=\ln\left(n\right)$, $s_{n}=\pi n\left(-1\right)^{n}$,
$\mathcal{N}\left(x\right)=1$, $N=20\text{ and }100$. For $x_{n}$
to be well-defined real numbers, only $n>0$ terms are considered
in the sum (\ref{eq:genWS}). The approximation is show in Fig. \ref{Fig:wsPics}(c).
Increasing the number of terms does not seem to help dumping oscillations
and improve the convergence.
\item [{(d)}] We approximate $f\left(x\right)=\ln\left(x^{2}+1\right)$
using $s(x)=1/x$, $x_{n}=1/n$, $s_{n}=\pi\left(-1\right)^{n+1}$,
$\mathcal{N}\left(x\right)=x^{2}$, $N=20$. For $x_{n}$ to be well-defined,
the term $n=0$ is excluded from the sum. The approximation is show
in Fig. \ref{Fig:wsPics}(d).
\item [{(e)}] We approximate $f\left(x\right)=\sqrt{1-x^{2}}$ using $s(x)=x/\left(1-x^{2}\right)$,
$x_{n}=2n/(\sqrt{4n^{2}+1}+1)$, $s_{n}=\pi\left(-1\right)^{n}\sqrt{4n^{2}+1}$,
$\mathcal{N}\left(x\right)=1-x^{2}$, $N=20\text{ and }100$. The
approximation is show in Fig. \ref{Fig:wsPics}(e).
\item [{(f)}] We approximate $f\left(x\right)=J_{0}\left(x\right)$ using
$s(x)=xe^{x^{2}}$, $x_{n}=\text{sgn}\left(n\right)\sqrt{W\left(2n^{2}\right)/2}$,
$\mathcal{N}\left(x\right)=1$, $N=20$, where $J_{0}$ is the Bessel
function, $\text{sgn}$ is the sign function, $W$ is the principal
branch of the Lambert W function and the expression for $s_{n}$ is
not shown (because of its complexity). The approximation is depicted
in Fig. \ref{Fig:wsPics}(f).
\end{lyxlist}
The general observations for the studied cases are
\begin{itemize}
\item Functions $\mathcal{A}_{N}^{f,WS}$ seem to converge (not necessarily
to $f$) and thus approximate $f$ at least to some extent. In the
case (a), the missing interpolation point $x_{0}=0$ situated in between
other points causes the approximation to deviate importantly in its
proximity.
\item It is difficult to assert about the convergence $\mathcal{A}_{N}^{f,WS}\rightarrow f$.
The cases (c) and (e) suggest that (for some functions) a higher approximation
order does not improve the convergence. Increasing $N$ however extends
the range $(x_{n}^{\text{min}},x_{n}^{\text{max}})$ and thus enlarges
the interval on which $\mathcal{A}_{N}^{f,WS}$ provides some (at
least rough) approximation of $f$.
\item In situations where $x_{n}$ points concentrate on the edges of the
$(x_{n}^{\text{min}},x_{n}^{\text{max}})$ interval, a Gibbs-like
phenomenon is observed there.
\end{itemize}
For integral matching the set of characteristic numbers is given by
\[
c_{n}^{f}=\int_{a}^{x_{n}}f\left(x\right)dx.
\]
One first needs to choose appropriate building blocs\footnote{For example $x\sin\left(\pi x\right)/\left[\left(-1\right)^{n}n\pi\left(x-n\right)\right]$
for $a=0$ and $x_{n}=n$.} for the interpolation function $\mathcal{A}^{f,\text{int.}}$ so
as to satisfy 
\[
A^{f,\text{int.}}\left(a\right)=0
\]
and thus ensure
\[
\left[\mathcal{A}^{f,\text{int.}}\left(x\right)\right]_{a}^{x_{n}}\equiv\mathcal{A}^{f,\text{int.}}\left(x_{n}\right)-\mathcal{A}^{f,\text{int.}}\left(a\right)=\mathcal{A}^{f,\text{int.}}\left(x_{n}\right)
\]
Interpreting $A^{f,\text{int.}}$ as the primitive function
\[
\mathcal{A}^{f,\text{int.}}\left(x_{n}\right)=F\left(x_{n}\right)=c_{n}^{f},\quad F'=f,
\]
the integral matching approximation $A^{f,\int}$ of $f$ is written
as 
\[
\mathcal{A}^{f,\int}\left(x\right)=\left[\mathcal{A}^{f,\text{int.}}\left(x\right)\right]'.
\]
The fact that the differentiation often generates oscillation \cite{berry2005universal}
seems to be a drawback of this method. The latter suggests that maybe
an opposite procedure should be performed: differentiate $f$ several
times, interpolate the resulting higher order derivative and then
apply (to the approximation) a repeated integration, which is known
to have (in general) a smoothing effect.

\section{Summary, conclusion, outlook}

The text interprets function approximation in a very general framework
of matching the characteristic numbers given by a (possibly non-linear)
functional action on the approximated function. To our knowledge all
existing approximations fit into this construction, we reviewed some
of them. Further, we proposed several new expansions mostly exploiting
the Taylor-like derivative-matching approach, but we also numerically
investigated some extensions of the Whittaker--Shannon interpolation
formula.

Maybe the most interesting results are the three presumably new rational
expansions (\ref{eq:newPade}), (\ref{eq:powInArg_Rat1}) and (\ref{eq:powInArg_Rat2}),
which posses interesting properties, such as an efficient evaluation
on a computer, an integrability within elementary functions or coefficients
which can be easily calculated order-by-order (unlike for the Padé
approximant). In addition, with the differentiation being linear,
one can consider an approximation by linearly combining them. Unfortunately,
one cannot fully control the positions of singularities since these
are given by construction.

The text also opens the possibility to further investigate new approximations
by focusing on other special cases of the Bell polynomial arguments
(Sec. \ref{subsec:Power-series-of}) or other pairs of mutually Dirichlet-inverse
arithmetic functions (Sec. \ref{subsec:g(x^n)}).

We did not provide detailed motivations to all new expansions. Yet,
several of them are constructed using a general function $g$ and
thus our results represent a large family of possible approximations.
A number of them can be well suited for some specific purposes the
author may not be aware of, yet we have many examples of mathematical
methods whose usefulness was seen only after their development.

In the present article we focused almost entirely on the approximation
property as we understand it, and did not make conclusions concerning
the convergence, unless such conclusions could be simply related to
known cases (the Taylor series). The convergence issues being usually
technically difficult, the text can be understood as a starting point
for studying them (for the various new expansions we proposed) in
the future.

\section*{Acknowledgments}

The work was supported by VEGA grant No. 2/0105/21.

\bibliographystyle{unsrt}
\bibliography{genExpansions}

\end{document}